\documentclass[nosumlimits,twoside]{amsart}
\usepackage{amsfonts, amsmath, amssymb}
\usepackage{graphicx, subcaption}
\usepackage{float}
\usepackage{srcltx}
\usepackage[all]{xy}
\usepackage{bbm}
\usepackage{version}
\usepackage{aligned-overset}
\usepackage[T1]{fontenc}
\usepackage{xcolor}
\usepackage{enumerate}
\usepackage[normalem]{ulem}
\usepackage{bm}
\usepackage{latexsym}
\usepackage{stmaryrd}
\usepackage[2emode]{psfrag}
\usepackage{yhmath}
\usepackage{array}
\usepackage{dsfont}
\usepackage{mathrsfs}
\usepackage{tikz-cd}
\usepackage{comment}
\usetikzlibrary { decorations.pathmorphing, decorations.pathreplacing, decorations.shapes }
\DeclareMathOperator{\rbiprod}{{\cdot\kern-.33em\triangleright\kern-.43em<}}
\DeclareMathOperator{\lbiprod}{{>\!\!\!\triangleleft\kern-.33em\cdot}}

\newtheoremstyle{mytheorem}%
{10.0pt plus 2.0pt minus 2.0pt} 
{10.0pt plus 2.0pt minus 2.0pt} 
{\itshape} 
{} 
{\bfseries} 
{.} 
{ } 
{} 

\newtheoremstyle{mydefinition}%
{10.0pt plus 2.0pt minus 2.0pt} 
{10.0pt plus 2.0pt minus 2.0pt} 
{} 
{} 
{\bfseries} 
{.} 
{ } 
{} 

\newtheoremstyle{myremark}%
{10.0pt plus 2.0pt minus 2.0pt} 
{10.0pt plus 2.0pt minus 2.0pt} 
{} 
{} 
{\itshape} 
{.} 
{ } 
{} 

\theoremstyle{mytheorem}
\newtheorem{theorem}{Theorem}[section]
\newtheorem{lemma}[theorem]{Lemma}

\newtheorem{corollary}[theorem]{Corollary}
\newtheorem{proposition}[theorem]{Proposition} 

\theoremstyle{myremark}
\newtheorem{remark}[theorem]{Remark}

\theoremstyle{mydefinition}
\newtheorem{definition}[theorem]{Definition}

\newtheorem{counterexample}[theorem]{Counterexample}
\newtheorem{problem}[theorem]{Problem}

\newtheoremstyle{myzusatz}
 {10.0pt plus 2.0pt minus 2.0pt} 
{10.0pt plus 2.0pt minus 2.0pt} 
{\itshape} 
{} 
{\bfseries} 
{.} 
{ } 
{\thmname{#1}\thmnumber{ #2}\thmnote{ #3}}

\theoremstyle{myzusatz}

\definecolor{gray1}{gray}{0.8}
\definecolor{gray2}{gray}{0.6}
\definecolor{gray3}{gray}{0.4}
\definecolor{gray4}{gray}{0.2}

\allowdisplaybreaks
\excludeversion{invisible?}
\excludeversion{proof?}
\newcommand{\id}{\mathrm{Id}}

\newcommand{\Mm}{\mathcal{M}}
\newcommand{\mm}{\mathfrak{M}}

\newcommand{\Rr}{\mathcal{R}}

\newcommand{\ot}{\otimes}

\renewcommand{\iff}{\Leftrightarrow}

\def\vect{{\mathfrak{M}}}

\newcommand{\C}{\mathbb{C}}

\newcommand{\op}{\mathrm{op}}

{
\left\{\begin{aligned}}
{\end{aligned}
\right.
}

\newcommand{\One}{\mathbbm{1}}

\begin{document}
\hyphenation{sa-cri-fi-cing}
\hyphenation{pro-duct}
\hyphenation{par-ti-cu-lar}
\title{Matched pairs and Yetter--Drinfeld braces}
\author{Davide Ferri, Andrea Sciandra}
\begin{abstract}\noindent
    It is proven that a matched pair of actions on a Hopf algebra $H$ is equivalent to the datum of a Yetter--Drinfeld brace, which is a novel structure generalising Hopf braces. This improves a theorem by Angiono, Galindo and Vendramin, originally stated for cocommutative Hopf braces. These Yetter--Drinfeld braces produce Hopf algebras in the category of Yetter--Drinfeld modules over $H$, through an operation that generalises Majid's transmutation. A characterisation of Yetter--Drinfeld braces via 1-cocycles, in analogy to the one for Hopf braces, is given.

    Every coquasitriangular Hopf algebra $H$ will be seen to yield a Yetter--Drinfeld brace, where the additional structure on $H$ is given by the transmutation. We compute explicit examples of Yetter--Drinfeld braces on the Sweedler's Hopf algebra, on the algebras $E(n)$, on $\mathrm{SL}_{q}(2)$, and an example in the class of Suzuki algebras.
\end{abstract}
\address{%
\parbox[b]{0.9\linewidth}{University of Turin, Department of Mathematics ``G.\@ Peano'',\\ via
 Carlo Alberto 10, 10123 Torino, Italy.\\
 Vrije Universiteit Brussel, Department of Mathematics and Data Science,\\ Pleinlaan 2, 1050, Brussels, Belgium.}}
 \email{d.ferri@unito.it, Davide.Ferri@vub.be}
\address{%
\parbox[b]{0.9\linewidth}{University of Turin, Department of Mathematics ``G.\@ Peano'',\\ via
 Carlo Alberto 10, 10123 Torino, Italy.}}
 \email{andrea.sciandra@unito.it}
\maketitle
\tableofcontents 
\section{Introduction}\noindent The notion of braiding is ubiquitous in Mathematics, encompassing low-dimension topology, group theory, category theory, and the theory of the Yang--Baxter equation (YBE). In particular, notions of braided categories \cite{joyal1993braided}, braided groups \cite{LYZ,Majid_braidedgroups, majid1995algebras}, and (co)quasitriangular Hopf algebras \cite{doi:10.1080/00927879308824649, drinfeld1989quasi, larson1991two,MJbraided,Sch} have been defined around the 1990s, and extensively used ever since.

On the other hand, set-theoretic solutions to the YBE have been studied on suggestion of Drinfeld \cite{drinfeld2006some}, growing into a vibrant field of research. It is known that braided groups provide solutions to the set-theoretic YBE, and solutions to the YBE are associated with a structure group \cite{etingof1999set} which turns out to be a braided group. Lu, Yan and Zhu \cite{LYZ} clarified that a braided group is essentially the same as a matched pair of actions. This is equivalent to bijective 1-cocycles. Afterwards, Guarnieri and Vendramin \cite{guarnieri2017skew} proved that this is in turn equivalent to a skew brace structure: a notion that has become of capital importance, in its several generalisations (a non-exhaustive list of them being provided by \cite{brzezinski2019trusses,MARTINLYONS2024751, matsu}). Skew braces generalise Rump's braces \cite{rump2007braces}, and the interplay between skew braces and Yang--Baxter maps generalises the correspondence between braces and involutive Yang--Baxter maps \cite{GATEVAIVANOVA2018649}. 

A skew brace is the datum of two group structures $(G,\cdot)$ and $(G,\bullet)$ on the same set $G$, satisfying a compatibility. An immediate Hopf-theoretic generalisation is the notion of Hopf brace, which is the datum of two Hopf algebra structures $(H,\cdot,1,\Delta, \epsilon, S)$ and $(H,\bullet, 1,\Delta, \epsilon, T)$ on the same coalgebra $(H,\Delta, \epsilon)$, satisfying an analogous compatibility \cite{angiono2017hopf}. The correspondence between skew braces, matched pairs of actions, and bijective 1-cocycles generalises naturally to cocommutative Hopf braces, cocommutative matched pairs on a Hopf algebra, and bijective 1-cocycles of cocommutative Hopf algebras. The correspondence between Hopf braces and bijective 1-cocycles is true without the assumption of cocommutativity---in fact, the two structures are strongly related, and the same equivalence holds true in any braided monoidal category \cite{eq1cocycleHopfbrace}. However, the assumption of cocommutativity is necessary for the equivalence with matched pairs; which is, in turn, crucial in order to bridge Hopf braces with the world of braiding operators. Thus, cocommutativity seemed hard to relax, without sacrificing the backbone of the correspondence theorem: this is the premise to our investigation.

In this paper, we get rid of the cocommutativity hypothesis. As in \cite{angiono2017hopf}, we consider matched pairs of Hopf algebras, satisfying a braided-commutativity condition: we call them \emph{matched pairs of actions on a Hopf algebra} (the name being reminiscent of matched pairs of actions on a group, extensively used by Lu, Yan and Zhu \cite{LYZ} with the name \emph{compatible actions}). Then, in {\S}$\,$\ref{section:constructing-hopf-in-YD}, we establish a correspondence between such matched pairs of actions, and Yetter--Drinfeld braces. A Yetter--Drinfeld brace is the datum of a Hopf algebra $H^\bullet = (H,\bullet, 1, \Delta, \epsilon, T)$, and a Hopf algebra $(H,\cdot,1, \Delta,\epsilon,S)$ in $^{H^\bullet}_{H^\bullet}\mathcal{YD}$ (which we call the transmutation of $H^\bullet$); where the two structures satisfy some technical hypotheses, and the same compatibility as Hopf braces. It comes out that the notions of Yetter--Drinfeld brace and of Hopf brace coincide in the cocommutative case.

It is natural to ask whether Yetter--Drinfeld braces can be interpreted as 1-cocycles as well. This is positively answered in {\S}$\,$\ref{section:1-cocycles}.

In {\S}$\,$\ref{section:coquasi}, we present a vast class of Yetter--Drinfeld braces that are, generally, not Hopf braces. Given any coquasitriangular Hopf algebra $(H,\bullet,1,\Delta,\epsilon, T,\Rr)$, we shall be able to construct a second operation $\cdot$ and a linear map $S\colon H\to H$, yielding a Yetter--Drinfeld brace. In this case, the additional structure on $H$ coincides with the transmutation of $H^\bullet$, introduced by Majid \cite{MJbraided, Maj2}. Therefore, we obtain in particular an interpretation of the theory of transmutation, within the frame of Yetter--Drinfeld braces. The map $\sigma\colon H\ot H\to H\ot H$, $\sigma(a\ot b):= \Rr^{-1}(a_1\ot b_1)b_2\ot a_2\Rr(a_3\ot b_3)$ is a braiding operator on $H$. When $(H^\bullet, \Rr)$ is cotriangular, the braiding operator $\sigma$ is involutive, and the corresponding algebra $(H,\cdot,1, \Delta, \epsilon, S)$ in $\mathrm{Hopf}({}^{H^\bullet}_{H^\bullet}\mathcal{YD})$ is braided-commutative with respect to the braiding of $^{H^\bullet}_{H^\bullet}\mathcal{YD}$.

This entire class of Yetter--Drinfeld braces descending from coquasitriangular Hopf algebras is, in some way, ``orthogonal'' to the class of cocommutative Hopf braces. When $(H^\bullet,\Rr)$ is a cocommutative coquasitriangular Hopf algebra, the corresponding Yetter--Drinfeld brace is indeed the trivial Hopf brace $(H,\bullet,\bullet)$.

In conclusion to this paper, {\S}$\,$\ref{section:example} presents several examples of coquasitriangular Hopf algebras, and describes their associated Yetter--Drinfeld braces.

\vskip 10pt

\noindent\textbf{Notations and conventions.} We shall denote by $(\Mm,\ot,\One,\sigma)$ a braided monoidal category, with monoidal product $\ot$, unit object $\One$, and braiding $\sigma$. Comforted by the MacLane Coherence Theorem, we shall consistently be sloppy on associativity and unit constraints. In the rest of this paper, we adopt the notation $fg$ for the composition of maps $f\circ g$. The categories of algebras, coalgebras, bialgebras and Hopf algebras in $\Mm$ will be denoted by $\mathrm{Mon}(\Mm)$, $\mathrm{Comon}(\Mm)$, $\mathrm{Bimon}(\Mm)$ and $\mathrm{Hopf}(\Mm)$, respectively. We shall usually denote by $H$ a bialgebra in $\Mm$, by ${}_H\!\Mm$, $\Mm_H$, ${}^H\!\Mm$, $\Mm^H$, ${}_H\!\Mm_H$ and ${}^H\!\Mm^H$ its category of left and right modules, left and right comodules, bimodules, and bicomodules, respectively. Unless otherwise specified, our setting will be the category $\vect=\mathsf{Vec}_\Bbbk$ of vector spaces over $\Bbbk$, with braiding given by the canonical flip $\tau$.

Symbols such as $\bullet,\cdot$ will usually signify the multiplication of an algebra $A$. Equivalently, the multiplication will be denoted by a map $m$, $m_\cdot$ or $m_\bullet\colon A\ot A\to A$. The unit is denoted by $1$ or, equivalently, by a map $u\colon \Bbbk \to A$. The comultiplication and the counit of a coalgebra will be denotedy by $\Delta$ and $\epsilon$, respectively. Subscripts will be added for clarity, whenever needed.

In our computations with coalgebras, we shall consistently employ \emph{Sweedler's notation}  $\Delta(a)=\sum_{i}a_1^i\ot a_2^i = a_1\ot a_2$, omitting the summation. A similar notation is adopted for left coactions $a\mapsto a_{-1}\ot a_0$, and for right coactions $a\mapsto a_0\ot a_1$.
\vskip 10pt

\section{Preliminaries}\label{section:prelim}\noindent
In this section, we introduce the main notions that we are going to use throughout the paper. Namely, we describe \emph{matched pairs of actions} on a bialgebra, and we survey the fundamentals on Hopf braces.
\subsection{Matched pairs of actions} 
Matched pairs of groups, enabling to define a Zappa--Szép product, have been known for a long time. On the other hand, the notion of a matched pair of Hopf algebras was introduced by Singer \cite{Singer} in the graded case, 
and the current notion was 
introduced by Majid \cite{majid1990physics}. We consider \emph{matched pairs of actions} on Hopf algebras, which will turn out to be a subclass of matched pairs of Hopf algebras \cite[Definition 7.2.1]{Majid-book}. Matched pairs of actions on groups have been already used, e.g. in the seminal work of Lu, Yan and Zhu \cite{LYZ}, and then mutuated in several different contexts, such as for groupoids by Andruskiewitsch \cite{andruskiewitsch2005quiver}. 
\begin{definition}\label{def:MPA}
    Let $H$ be a bialgebra, $\bullet$ be its algebra product. A \emph{matched pair of actions} $(H,\rightharpoonup,\leftharpoonup)$ on $H$ is the datum of a left action $\rightharpoonup$ and a right action $\leftharpoonup$ of $H$ on itself, such that $H$ is a left $H$-module coalgebra and a right $H$-module coalgebra with the respective actions, and the following conditions hold for all $a,b,c\in H$:
    \begin{align}
        \label{matchedpair-1}\tag{\sc mp.1}& a\rightharpoonup 1 = \epsilon(a)1,\text{ i.e.\@ }u_H\text{ is a morphism in }{}_H\vect;\\
        \label{matchedpair-2}\tag{\sc mp.2}&1\leftharpoonup a = \epsilon(a)1,\text{ i.e.\@ }u_H\text{ is a morphism in }{}\vect_H;\\
        \label{matchedpair-3}\tag{\sc mp.3}& a\rightharpoonup(b\bullet c)=(a_{1}\rightharpoonup b_{1})\bullet ((a_{2}\leftharpoonup b_{2})\rightharpoonup c);\\
        \label{matchedpair-4}\tag{\sc mp.4}& (a\bullet b)\leftharpoonup c=(a\leftharpoonup(b_{1}\rightharpoonup c_{1}))\bullet (b_{2}\leftharpoonup c_{2});\\
        &\label{braided-commutativity}\tag{$\star$} a\bullet b = (a_1\rightharpoonup b_1)\bullet (a_2\leftharpoonup b_2).
    \end{align}
    A \emph{morphism of matched pairs of actions} between $(H,\rightharpoonup_H, \leftharpoonup_H)$ and $(K,\rightharpoonup_K, \leftharpoonup_K)$ is a morphism of bialgebras $H\to K$ that intertwines the two left actions and the two right actions, respectively, in the following sense:
    \[f(a\rightharpoonup_H b) = f(a)\rightharpoonup_K f(b),\quad f(a\leftharpoonup_H b) = f(a)\leftharpoonup_K f(b). \] The category of matched pairs of actions (in $\mathsf{Vec}_\Bbbk$) will be denoted by $\mathrm{MP}(\mathsf{Vec}_\Bbbk)$.
\end{definition}


A pair of actions $(\rightharpoonup,\leftharpoonup)$ on a bialgebra $H$ satisfying \eqref{matchedpair-1}--\eqref{matchedpair-4} allows one to construct the \emph{double cross product} $H \bowtie H$, as described in Majid \cite{Majid-book}, which is the vector space $H\ot H$ endowed with the algebra structure given by the product \[(a\ot h)(b\ot g):=a\bullet (h_{1}\rightharpoonup b_{1})\ot(h_{2}\leftharpoonup b_{2})\bullet g,\] with unit $1_{H}\ot1_{H}$, and the usual tensor product coalgebra structure. Then $H\bowtie H$ becomes a bialgebra if and only if the additional condition
    \begin{equation}
        \label{matchedpair5}\tag{\sc mp.5} (a_1\rightharpoonup b_1)\ot (a_2\leftharpoonup b_2) = (a_2\rightharpoonup b_2)\ot (a_1\leftharpoonup b_1)
    \end{equation}
    holds; and in this case, $(H,H)$ is a \emph{matched pair of bialgebras} in the sense of Majid \cite[Definition 7.2.1]{Majid-book}. Notice that \eqref{matchedpair5} is trivially satisfied in case $H$ is cocommutative.

        When $H$ is a Hopf algebra, a matched pair of bialgebras $(H,H)$ will be usually called a \emph{matched pair} on $H$.    
\begin{remark}\label{rem:trivialactions}
        For any bialgebra $H$, the actions $a\rightharpoonup b := \epsilon(a)b$, $a\leftharpoonup b := a\,\epsilon (b)$ always satisfy \eqref{matchedpair-1}--\eqref{matchedpair5}. They moreover satisfy \eqref{braided-commutativity} if and only if $\bullet$ is commutative. Notice that $\eqref{braided-commutativity}$ is equivalent to asking that $m_\bullet \colon H\bowtie H\to H$ is a morphism of algebras. 
        \end{remark}
The following result simply follows by an argument of naturality.

\begin{lemma}\label{lem:mhex}
    Let $\Mm$ be a (strict) braided monoidal category with braiding $\sigma$, let $H$ be an object in $\Mm$, and $m\colon H\ot H\to H$, $u\colon \One\to H$ any morphisms in $\mathcal{M}$. Then, the following hold:
    \begin{align}
        \label{braided1}\tag{\sc br.1}&\sigma_{H,H}(\id\ot u) = u\ot \id;\\
        \label{braided2}\tag{\sc br.2}&\sigma_{H,H}(u\ot \id) = \id\ot u;\\
        \label{braided3}\tag{\sc br.3}&\sigma_{H,H} m_{23} =m_{12}(\sigma_{H,H})_{23}(\sigma_{H,H})_{12};\\
        \label{braided4}\tag{\sc br.4}&\sigma_{H,H} m_{12} =m_{23}(\sigma_{H,H})_{12}(\sigma_{H,H})_{23}.
    \end{align}
Moreover, if $m$ is a monomorphism in $\mathcal{M}$ and $\varsigma\colon H\ot H\to H\ot H$ is a morphism satisfying \eqref{braided1}--\eqref{braided4}, then $\varsigma$ satisfies the hexagonal axioms, and hence $(H,\varsigma)$ is a braided object in $\Mm$.
\end{lemma}

\begin{lemma}[{cf.\@ Tambara \cite[Proposition 2.2]{tambara1990coendomorphism}}]\label{lemma:matchedpair-iff-braiding}
    Let $H$ be a bialgebra. Suppose given a left action $\rightharpoonup:H\otimes H\to H$ and a right action $\leftharpoonup:H\otimes H\to H$, such that $H$ is a left $H$-module coalgebra and a right $H$-module coalgebra with the respective actions. Define $\sigma_{H,H}\colon H\ot H\to H\ot H$ by
\[
\sigma_{H,H}(a\otimes b):=(a_{1}\rightharpoonup b_{1})\otimes(a_{2}\leftharpoonup b_{2}).
\]
Then, $\sigma_{H,H}$ satisfies \eqref{braided1}--\eqref{braided4} if and only if $\rightharpoonup$ and $\leftharpoonup$ satisfy \eqref{matchedpair-1}--\eqref{matchedpair-4}.
\end{lemma}
\begin{proof}
One has 
\[
\sigma_{H,H}m_{12}(a\otimes b\otimes c)=\sigma_{H,H}(ab\otimes c)=(a_{1}b_{1}\rightharpoonup c_{1})\otimes(a_{2}b_{2}\leftharpoonup c_{2}),
\]
while on the other hand
\[
\begin{split}
&\hspace{-2em}m_{23}(\sigma_{H,H})_{12}(\sigma_{H,H})_{23}(a\otimes b\otimes c)\\
&=m_{23}(\sigma_{H,H})_{12}\big(a\otimes(b_{1}\rightharpoonup c_{1})\otimes(b_{2}\leftharpoonup c_{2})\big)\\&=m_{23}\big((a_{1}\rightharpoonup(b_{1}\rightharpoonup c_{1})_{1})\otimes(a_{2}\leftharpoonup(b_{1}\rightharpoonup c_{1})_{2})\otimes(b_{2}\leftharpoonup c_{2})\big)\\&=(a_{1}\rightharpoonup(b_{1}\rightharpoonup c_{1}))\otimes(a_{2}\leftharpoonup(b_{2}\rightharpoonup c_{2}))(b_{3}\leftharpoonup c_{3})\\&=(a_{1}b_{1}\rightharpoonup c_{1})\otimes(a_{2}\leftharpoonup(b_{2}\rightharpoonup c_{2}))(b_{3}\leftharpoonup c_{3}).
\end{split}
\]
If \eqref{braided4} holds true then, by applying $\epsilon\ot\id$ on both sides, we obtain $(ab)\leftharpoonup c=(a\leftharpoonup(b_{1}\rightharpoonup c_{1}))(b_{2}\leftharpoonup c_{2})$, which is \eqref{matchedpair-4}. If conversely \eqref{matchedpair-4} holds true, then $\sigma_{H,H}$ clearly satisfies \eqref{braided4}.
One analogously proves that \eqref{braided3} is equivalent to \eqref{matchedpair-3}.
Observe that
\[
\sigma_{H,H}(1\otimes a)=(1\rightharpoonup a_{1})\otimes(1\leftharpoonup a_{2})=a_{1}\otimes(1\leftharpoonup a_{2}).
\]
If \eqref{braided2} holds true, then
\[
a_{1}\otimes(1\leftharpoonup a_{2})=a\otimes1.
\]
Hence, by applying $\epsilon\ot\id$, we obtain $1\leftharpoonup a=\epsilon(a)1$, which is \eqref{matchedpair-2}. Conversely, if \eqref{matchedpair-2} holds true, then $\sigma_{H,H}$ clearly satisfies \eqref{braided2}. The equivalence between \eqref{braided1} and \eqref{matchedpair-1} is proven analogously.
\end{proof}
\begin{remark}\label{corollary:twoactions}
Let $H$ be a bialgebra, and suppose given
\[
\sigma_{H,H}(x\otimes y)=(x_{1}\rightharpoonup y_{1})\otimes(x_{2}\leftharpoonup y_{2})
\]
where $\rightharpoonup:H\ot H\to H$ and $\leftharpoonup:H\ot H\to H$ are such that $\epsilon(a\rightharpoonup b)=\epsilon(a)\epsilon(b)=\epsilon(a\leftharpoonup b)$ for all $a,b\in H$. Then, $\rightharpoonup$ is retrieved as $(\id\ot\epsilon)\sigma_{H,H}$, and $\leftharpoonup$ as $(\epsilon\ot\id)\sigma_{H,H}$ (see \cite[p.\@ 300]{Majid-book}).
\end{remark}
Although there are examples of braidings that are obtained from a pair of actions satisfying \eqref{matchedpair-1}--\eqref{matchedpair-4} (see {\S}$\,$\ref{subsection:coQT-and-MP}), it is not true that all braidings are obtained from a pair of maps as in Remark \ref{corollary:twoactions}:
\begin{counterexample}\label{counterexample:noharpoons}
 Consider, for instance, a coquasitriangular bialgebra $(H,\Rr)$ (see {\S}$\,$\ref{subsection:coquasi}) and the usual braiding \eqref{braiding-comodules} on $\vect^H$
\[
\sigma_{H,H}\colon H\otimes H\to H\otimes H,\quad a\ot b\mapsto b_{1}\ot a_{1}\Rr(a_{2}\ot b_{2}).
\]
If we suppose by contradiction that $\sigma_{H,H}(a\ot b) = (a_1\rightharpoonup b_1)\ot (a_2\leftharpoonup b_2)$ for some maps $\rightharpoonup,\leftharpoonup$, then Remark \ref{corollary:twoactions} yields
\[
a\rightharpoonup b=\Rr(a\otimes b_{2})b_{1},\qquad a\leftharpoonup b=a_{1}\Rr(a_{2}\otimes b),
\]
but then
\begin{align*}
    b_1\ot a_1\Rr(a_2\ot b_2)&=\sigma_{H,H}(a\ot b)\\
    &= (a_1\rightharpoonup b_1)\ot (a_2\leftharpoonup b_2)\\
    &= \Rr(a_1\ot b_2)b_1\ot a_2\Rr(a_3\ot b_3),
\end{align*}
and this is generally false. As an example, consider the coquasitriangular structure on the Sweedler's Hopf algebra described in {\S}$\,$\ref{subsection:Sweedler}, with $a=b=g$: one has $\Rr(g\ot g) = -1$, whence $b_1\ot a_1\Rr(a_2\ot b_2) = -g\ot g$ and $\Rr(a_1\ot b_2)b_1\ot a_2\Rr(a_3\ot b_3) = g\ot g$, which differ if $\mathrm{char}(\Bbbk)\neq 2$.
\end{counterexample}

\begin{lemma}[{cf. Majid \cite[proof of Theorem 7.2.3]{Majid-book}}]\label{lemma:retrieve-actions}
    Given a bialgebra $H$ and a morphism $\sigma_{H,H}:H\otimes H\to H\otimes H$ satisfying \eqref{braided1}--\eqref{braided4}, the morphisms $\rightharpoonup\,:=(\mathrm{Id}\otimes\epsilon)\sigma_{H,H}$ and $\leftharpoonup\,:=(\epsilon\otimes\mathrm{Id})\sigma_{H,H}$ are left and right $H$-actions, respectively.
\end{lemma}
\begin{proof}
    We compute
\[
1\rightharpoonup c=(\mathrm{Id}\otimes\epsilon)\sigma_{H,H}(1\otimes c)=(\mathrm{Id}\otimes\epsilon)(c\otimes 1)=c,
\]
and also
\begin{align*}
    ab\rightharpoonup c&=(\mathrm{Id}\otimes\epsilon)\sigma_{H,H}m_{12}(a\otimes b\otimes c)\\&=(\mathrm{Id}\otimes\epsilon)m_{23}(\sigma_{H,H})_{12}(\sigma_{H,H})_{23}(a\otimes b\otimes c)\\&=(\mathrm{Id}\otimes\epsilon\otimes\epsilon)(\sigma_{H,H})_{12}(\sigma_{H,H})_{23}(a\otimes b\otimes c)\\&=(\mathrm{Id}\otimes\epsilon)\sigma_{H,H}(\mathrm{Id}\otimes\mathrm{Id}\otimes\epsilon)(\mathrm{Id}\otimes\sigma_{H,H})(a\otimes b\otimes c)\\&=(\mathrm{Id}\otimes\epsilon)\sigma_{H,H}(a\otimes(b\rightharpoonup c))\\&=a\rightharpoonup(b\rightharpoonup c).
\end{align*}
The proof is analogous for $\leftharpoonup$.
\end{proof}
In general, when $\sigma_{H,H}$ does not have the form $\sigma_{H,H}(x\otimes y)=(x_{1}\rightharpoonup y_{1})\otimes(x_{2}\leftharpoonup y_{2})$, the actions $\rightharpoonup$ and $\leftharpoonup$ defined as in the previous lemma need not form a matched pair, as the following counterexample shows.
\begin{counterexample}\label{counterexample:nomatchedpair}
Let us consider again
\[
\sigma_{H,H}\colon H\otimes H\to H\otimes H,\quad a\ot b\mapsto b_{1}\ot a_{1}\Rr(a_{2}\ot b_{2})
\]
and the left and right $H$-actions
\[
a\rightharpoonup b=\Rr(a\otimes b_{2})b_{1},\qquad a\leftharpoonup b=a_{1}\Rr(a_{2}\otimes b).
\]
Then, we have $a\rightharpoonup bc=\Rr(a\ot b_{2}c_{2})b_{1}c_{1}$, while
\[
(a_{1}\rightharpoonup b_{1})((a_{2}\leftharpoonup b_{2})\rightharpoonup c)=\Rr(a_{1}\ot b_{2})\Rr(a_{2}\ot c_{2})\Rr(a_{3}\ot b_{3})b_{1}c_{1}=\Rr(a\ot b_{3}c_{2}b_{2})b_{1}c_{1}.
\]
These are generally distinct (consider for instance the Sweedler's Hopf algebra in {\S}$\,$\ref{subsection:Sweedler}, $a = b = c=g$, $\mathrm{char}(\Bbbk)\neq 2$).
\end{counterexample}

\subsection{Hopf braces} A \emph{skew brace}, defined by Guarnieri and Vendramin \cite{guarnieri2017skew} generalising Rump \cite{rump2007braces}, is the datum of two group operations on the same set, satisfying a compatibility. Skew braces are ubiquitous in the theory of braidings and the Yang--Baxter equation, and hence yearned for a Hopf-theoretic version. This was defined by Angiono, Galindo and Vendramin \cite{angiono2017hopf}:
\begin{definition}
    A \emph{Hopf brace} $(H,\cdot,\bullet,1,\Delta, \epsilon,S,T)$ is the datum of two Hopf algebra structures $(H,\cdot,1,\Delta,\epsilon,S)$ and $(H,\bullet,1',\Delta,\epsilon,T)$ on the same coalgebra $(H,\Delta,\epsilon)$, satisfying the following compatibility:
    \begin{equation}
        \label{compHopfbraces}\tag{\scshape hbc}
    a\bullet(b\cdot c)=(a_{1}\bullet b)\cdot S(a_{2})\cdot(a_{3}\bullet c).
    \end{equation}
    Notice that the compatibility condition forces $1 = 1'$ (see \cite[Remark 1.3]{angiono2017hopf}).
\end{definition}
When no confusion arises, we shall simply indicate a Hopf brace by $(H,\cdot, \bullet)$. We recall the following characterisation:
\begin{proposition}[{Angiono, Galindo and Vendramin \cite[Theorems 1.12, 3.3, and Corollary 2.4]{angiono2017hopf}}]\label{Thm:Hopfbracecond.eq.} Given a Hopf algebra $H$, the following data are equivalent:
\begin{enumerate}
    \item[\textit{i.}] a Hopf brace structure $(H,\cdot,\bullet)$ on $H$;
    \item[\textit{ii.}] a Hopf algebra $A$, an action $\rightharpoonup$ of $A$ on $H$ such that $H$ is a left $A$-module algebra, and an isomorphism of coalgebras $\pi\colon A\to H$ which is a 1-cocycle of bialgebras; i.e., $\pi(ab) = \pi(a_1)(a_2\rightharpoonup \pi(b))$.
\end{enumerate} 
    In case $H$ is \emph{cocommutative}, the previous two are also equivalent to
\begin{enumerate}
   \item[\textit{iii.}] a matched pair on $H$, satisfying the additional condition \eqref{braided-commutativity}.
\end{enumerate}
Moreover, in the cocommutative case,
\[
\sigma_{H,H}\colon  a\otimes b \mapsto (a_{1}\rightharpoonup b_{1})\ot(a_{2}\leftharpoonup b_{2})
\]
is a coalgebra isomorphism and a solution to the braid equation on $H$.
\end{proposition}
The additional condition \eqref{braided-commutativity} is the \emph{braided commutativity} of $\bullet$ with respect to the braiding operator $\sigma_{H,H}$.
\begin{remark}
    In the cocommutative case, the matched pair obtained from a Hopf brace $(H,\cdot,\bullet)$ is given by $h\rightharpoonup k=S(h_{1})\cdot(h_{2}\bullet k)$ and $h\leftharpoonup k=T(h_{1}\rightharpoonup k_{1})\bullet h_{2}\bullet k_{2}$. Then, the corresponding solution of the braid equation is
\[
\sigma_{H,H}(x\otimes y)=\big(S(x_{1})\cdot(x_{2}\bullet y_{1})\big)\otimes\big(T(S(x_{3})\cdot(x_{4}\bullet y_{2}))\bullet x_{5}\bullet y_{3}\big).
\]
\end{remark}
    
\section{Matched pairs of actions and Yetter--Drinfeld braces}\label{section:constructing-hopf-in-YD}
\noindent
In this section, we generalise the correspondence of Proposition \ref{Thm:Hopfbracecond.eq.} by dropping the hypothesis of cocommutativity. As a result, we shall obtain a correspondence between matched pairs of actions on a Hopf algebra, and novel structures which we call \emph{Yetter--Drinfeld braces}.

\subsection{Yetter--Drinfeld braces} Let $(H,\bullet,1,\Delta, \epsilon, T) $ be a Hopf algebra, $(H,\rightharpoonup,\leftharpoonup)$ be a matched pair of actions on $H$. We begin by defining a second operation $\cdot$ and a map $S\colon H\to H$, mimicking \cite[Proposition 3.2]{angiono2017hopf}:
\begin{align}
    \label{def-cdot}& a\cdot b:= a_1\bullet (T(a_2)\rightharpoonup b);\\
    \label{def-S}& S(a):= a_1\rightharpoonup T(a_2).
\end{align}
Notice that $\bullet$ is retrieved from $\cdot$ as
\begin{equation}\label{compblulletdotharp}
a\bullet b= a_{1}\bullet(T(a_{2})\rightharpoonup(a_{3}\rightharpoonup b))= a_{1}\cdot(a_{2}\rightharpoonup b).
\end{equation}
Recall that the linearity of $\Delta$ and $\epsilon$ with respect to $\rightharpoonup$ read as:
\[
\Delta(a\rightharpoonup b)=(a_{1}\rightharpoonup b_{1})\otimes(a_{2}\rightharpoonup b_{2}),\qquad \epsilon(a\rightharpoonup b)=\epsilon(a)\epsilon(b),
\]
similarly for the linearity of $\Delta$ and $\epsilon$ with respect to $\leftharpoonup$.
\begin{remark}
    The condition \eqref{braided-commutativity} allows one to retrieve the right action from the left action, and vice versa. Indeed, by the linearity of $\Delta$ and $\epsilon$ with respect to $\rightharpoonup$, \eqref{braided-commutativity} is equivalent to
    \begin{align}
        \label{rightaction-from-leftaction} &a\leftharpoonup b = T(a_1\rightharpoonup b_1)\bullet a_2\bullet b_2.
      \end{align}
Moreover, \eqref{rightaction-from-leftaction} is clearly equivalent to      
      \begin{align}
        \label{T-on-harpoon}&T(a\rightharpoonup b)=(a_{1}\leftharpoonup b_{1})\bullet T(b_{2})\bullet T(a_{2}).
    \end{align}
From \eqref{T-on-harpoon} one can easily deduce
\begin{equation}\label{T-on-harpoon-with-T}
    T(a_{1}\rightharpoonup T(a_{2}))=a_{1}\leftharpoonup T(a_{2}).
\end{equation}
\end{remark}

\begin{remark}\label{S2}
    Let us observe that $S^{2}$ is given on an element $a\in H$ by
\begin{align*}
S^{2}(a)&=S(a_{1}\rightharpoonup T(a_{2}))\\&=(a_{1}\rightharpoonup T(a_{4}))\rightharpoonup T(a_{2}\rightharpoonup T(a_{3}))\\
\overset{\eqref{T-on-harpoon-with-T}}&{=}(a_{1}\rightharpoonup T(a_{4}))\rightharpoonup(a_{2}\leftharpoonup T(a_{3})).
\end{align*}
\end{remark}

\begin{lemma}\label{lem:deltadot}
    One has $\Delta (a\cdot b) = (a_{1}\bullet(T(a_{3})\rightharpoonup b_{1}))\ot(a_{2}\cdot b_{2})$.
\end{lemma}
\begin{proof}We compute
    \begin{align*}
    \Delta(a\cdot b)&=\Delta(a_{1}\bullet(T(a_{2})\rightharpoonup b))\\&=(a_{1}\bullet(T(a_{3})\rightharpoonup b)_{1})\ot(a_{2}\bullet(T(a_{3})\rightharpoonup b)_{2})
    \\&=(a_{1}\bullet(T(a_{4})\rightharpoonup b_{1}))\ot(a_{2}\bullet(T(a_{3})\rightharpoonup b_{2}))\\&=(a_{1}\bullet(T(a_{3})\rightharpoonup b_{1}))\ot(a_{2}\cdot b_{2}).
\end{align*}
\end{proof}
Although $\cdot$ is not a morphism of coalgebras in $\mathsf{Vec}_\Bbbk$, this will be sorted out by setting ourselves in a suitable category.

An immediate computation shows the form of the twice iterated $\Delta$, which is going to come in handy later:
\begin{equation}
    \label{triple-delta} \begin{split}
    (\Delta\ot\id)\Delta\,(a\cdot b)&=\Delta(a_{1}\bullet(T(a_{3})\rightharpoonup b_{1}))\otimes(a_{2}\cdot b_{2})\\&=
(a_{1}\bullet(T(a_{5})\rightharpoonup b_{1}))\otimes(a_{2}\bullet(T(a_{4})\rightharpoonup b_{2}))\otimes(a_{3}\cdot b_{3})\\&=(a_{1}\bullet(T(a_{6})\rightharpoonup b_{1}))\otimes(a_{2}\bullet(T(a_{5})\rightharpoonup b_{2}))\otimes(a_{3}\bullet(T(a_{4})\rightharpoonup b_{3})).
\end{split}
\end{equation}
\begin{proposition}\label{proposition:algebra}
    Let $(H,\bullet,1,\Delta,\epsilon,T)$ be a Hopf algebra, $(H,\rightharpoonup,\leftharpoonup)$ a matched pair of actions. Define $\cdot$ and $S$ as in \eqref{def-cdot}, \eqref{def-S}. Then, $(H,\cdot,1)$ becomes an algebra, $\epsilon$ is a morphism of algebras with respect to this structure, and $S$ satisfies $a_{1}\cdot S(a_{2})=\epsilon(a)1=S(a_{1})\cdot a_{2}$. Moreover, one has
    \begin{equation}
    \label{leftantipode} a\rightharpoonup b=S(a_{1})\cdot(a_{2}\bullet b)
    \end{equation}
and the two operations $\bullet$ and $\cdot$ satisfy the Hopf brace compatibility \eqref{compHopfbraces}.
\end{proposition}
\begin{proof}
We first observe that
\[
a\cdot1=a_{1}\bullet(T(a_{2})\rightharpoonup 1)=a_{1}\bullet\epsilon(T(a_{2}))1=a_{1}\epsilon(a_{2})=a,\qquad 1\cdot a=1\bullet(T(1)\rightharpoonup a)=a,
\]
for all $a\in H$. We now prove that $\cdot$ is associative. Given $a,b,c\in H$ we have: 
\begin{align*}
    a\cdot (b\cdot c)&= a\cdot \big(b_1\bullet (T(b_2)\rightharpoonup c)\big)\\
    &= a_1\bullet \Big( T(a_2)\rightharpoonup \big(b_1\bullet (T(b_2)\rightharpoonup c)\big) \Big)\\\overset{\eqref{matchedpair-3}}&{=} a_1\bullet \Big( T(a_3)\rightharpoonup b_1 \Big)\bullet \Big( (T(a_2)\leftharpoonup b_2)\rightharpoonup (T(b_3)\rightharpoonup c)\Big)\\\overset{\eqref{rightaction-from-leftaction}}&{=} a_1\bullet (T(a_4)\rightharpoonup b_1)\bullet \Big( T\big(T(a_3)\rightharpoonup b_2\big)\rightharpoonup \big( T(a_2)\rightharpoonup (b_3\rightharpoonup (T(b_4)\rightharpoonup c)) \big) \Big)\\
&= a_1\bullet (T(a_4)\rightharpoonup b_1)\bullet \Big( T\big(T(a_3)\rightharpoonup b_2\big)\rightharpoonup \big( T(a_2)\rightharpoonup c \big) \Big)\\
&= a_1\bullet (T(a_4)\rightharpoonup b_1)
    \bullet \Big(T\big(a_2\bullet (T(a_3)\rightharpoonup b_2)\big)\rightharpoonup c\Big)\\
    &= a_1\bullet (T(a_3)\rightharpoonup b_1)\bullet \Big(T(a_2\cdot b_2)\rightharpoonup c\Big)\\
    &= (a\cdot b)_1\bullet \Big(T\big((a\cdot b)_2\big)\rightharpoonup c\Big)\\
    &= (a\cdot b)\cdot c,
    \end{align*}
as desired. We then compute
\[
\epsilon(a\cdot b)=\epsilon(a_{1}\bullet(T(a_{2})\rightharpoonup b))=\epsilon(a_{1})\epsilon(T(a_{2})\rightharpoonup b)=\epsilon(T(a))\epsilon(b)=\epsilon(a)\epsilon(b).
\]
Moreover, one has
\begin{align*}
    a_{1}\cdot S(a_{2})&=a_{1}\bullet( T(a_{2})\rightharpoonup S(a_{3}))\\&=a_{1}\bullet(T(a_{2})\rightharpoonup(a_{3}\rightharpoonup T(a_{4})))\\&=a_{1}\bullet(T(a_{2})\bullet a_{3}\rightharpoonup T(a_{4}))\\&=a_{1}\bullet T(a_{2})\\&=\epsilon(a)1
\end{align*}
and 
\begin{align*}
  S(a_{1})\cdot a_{2}&=S(a_{1})_{1}\bullet(T(S(a_{1})_{2})\rightharpoonup a_{2})\\&=(a_{1_{1}}\rightharpoonup T(a_{2})_{1})\bullet(T(a_{1_{2}}\rightharpoonup T(a_{2})_{2})\rightharpoonup a_{3})\\&=(a_{1}\rightharpoonup T(a_{4}))\bullet(T(a_{2}\rightharpoonup T(a_{3}))\rightharpoonup a_{5})\\
  \overset{\eqref{T-on-harpoon-with-T}}&{=}(a_{1}\rightharpoonup T(a_{4}))\bullet\Big(\big(a_{2}\leftharpoonup T(a_{3})\big)\rightharpoonup a_{5}\Big)\\&=(a_{1_{1}}\rightharpoonup T(a_{2})_{1})\bullet\Big(\big(a_{1_{2}}\leftharpoonup T(a_{2})_{2}\big)\rightharpoonup a_{3}\Big)\\\overset{\eqref{matchedpair-3}}&{=}a_{1}\rightharpoonup(T(a_{2})\bullet a_{3})\\&=\epsilon(a)1.
\end{align*}
As for \eqref{leftantipode}, observe that
\[
a\rightharpoonup b=S(a_{1})\cdot a_{2}\cdot(a_{3}\rightharpoonup b)\overset{\eqref{compblulletdotharp}}{=}S(a_{1})\cdot(a_{2}\bullet b).
\]
We finally prove the Hopf brace compatibility \eqref{compHopfbraces}:
\begin{align*}
    a\bullet(b\cdot c)&=a\bullet b_{1}\bullet(T(b_{2})\rightharpoonup c)\\&=a_{1}\bullet b_{1}\bullet(T(b_{2})\bullet T(a_{2})\bullet a_{3}\rightharpoonup c)\\&=a_{1}\bullet b_{1}\bullet(T(a_{2}\bullet b_{2})\rightharpoonup(a_{3}\rightharpoonup c))\\&=(a_{1}\bullet b)\cdot(a_{2}\rightharpoonup c)\\\overset{\eqref{leftantipode}}&{=}(a_{1}\bullet b)\cdot S(a_{2})\cdot(a_{3}\bullet c).
\end{align*}
\end{proof}
As promised, we now reinterpret this additional operation $\cdot$ in a suitable category, which is going to be the category of \emph{Yetter--Drinfeld modules} over $H$.
\begin{definition}
    Let $(H,\bullet,1,\Delta, \epsilon, T)$ be a Hopf algebra. A (\emph{left-left}) \emph{Yetter--Drinfeld module} on $H$ is the datum of a left $H$-module $X$ with an action $\rightharpoonup\colon H\ot X\to X$, which is also a left $H$-comodule with a coaction $\rho\colon X\to H\ot X$, satisfying the following compatibility for all $x\in X$, $a\in H$:
    \begin{equation}
        \label{YD}\tag{\sc yd} \rho(a\rightharpoonup x)=a_{1}\bullet x_{-1}\bullet T(a_{3})\otimes(a_{2}\rightharpoonup x_{0}).
    \end{equation}
    A morphism of Yetter--Drinfeld modules is a morphism of both left $H$-modules and left $H$-comodules. We denote by $^H_H\mathcal{YD}$ the category of Yetter--Drinfeld modules over $H$. 
\end{definition}

\begin{lemma}\label{lemma:equiv-cond-colinearity}
    Let $(H,\bullet,1,\Delta,\epsilon,T)$ be a Hopf algebra. Suppose that $\rightharpoonup\colon H\ot H\to H$ is a morphism of coalgebras, $\leftharpoonup\colon H\ot H\to H$ is a linear map satisfying $\epsilon(a\leftharpoonup b)=\epsilon(a)\epsilon(b)$, and \eqref{braided-commutativity} is satisfied. The following conditions are equivalent, and they are all equivalent to \eqref{matchedpair5}\textup{:}\begin{align}
    &\label{delta-colinear}(a\leftharpoonup b)_1\ot (a\leftharpoonup b)_2 = (a_1\leftharpoonup b_1)\ot (a_2\leftharpoonup b_2),\text{ i.e.\@ }\leftharpoonup\text{ is a morphism of coalgebras}\,;\\
    &\label{compharpoonbullet3} \big((a_1\rightharpoonup b_1)\bullet (a_3\leftharpoonup b_3)\big)\ot \big((a_2\rightharpoonup b_2)\bullet (a_4\leftharpoonup b_4)\big)\\
    \nonumber& \hspace{2em}= \big((a_1\rightharpoonup b_1)\bullet (a_3\leftharpoonup b_3)_1\big)\ot \big((a_2\rightharpoonup b_2)\bullet (a_3\leftharpoonup b_3)_2\big);\\
    &\label{compharpoonbullet2} (a_2\leftharpoonup b_2)\ot \big((a_1\rightharpoonup b_1)\bullet (a_3\leftharpoonup b_3)\big) = (a_1\leftharpoonup b_1)\ot (a_2\bullet b_2);\\
    \label{compharpoonpbullet}
    &\big((a_{1}\rightharpoonup b_{1})\bullet(a_{3}\leftharpoonup b_{3})\big)\ot(a_{2}\rightharpoonup b_{2})=(a_{1}\bullet b_{1})\ot(a_{2}\rightharpoonup b_{2});
\end{align}
In particular, having a matched pair of actions $(H,\rightharpoonup,\leftharpoonup)$ all the previous equalities hold true.
\end{lemma}
\begin{proof}
   First, we notice that
    \begin{align*}
    &\hspace{-2em}\big((a_1\rightharpoonup b_1)\bullet (a_3\leftharpoonup b_3)_1\big)\ot \big((a_2\rightharpoonup b_2)\bullet (a_3\leftharpoonup b_3)_2\big)\\
    &= \big((a_1\rightharpoonup b_1)_1\bullet (a_2\leftharpoonup b_2)_1\big)\ot \big((a_1\rightharpoonup b_1)_2\bullet (a_2\leftharpoonup b_2)_2\big)\\
    &= \big((a_1\rightharpoonup b_1)\bullet (a_2\leftharpoonup b_2)\big)_1\ot \big((a_1\rightharpoonup b_1)\bullet (a_2\leftharpoonup b_2)\big)_2\\
    \overset{\eqref{braided-commutativity}}&{=} (a\bullet b)_1\ot (a\bullet b)_2\\
    &= (a_{1}\bullet b_{1})\ot(a_2\bullet b_2).
    \end{align*}
Now, we check that \eqref{compharpoonbullet3} and \eqref{delta-colinear} are equivalent:
    \begin{align*}
\eqref{compharpoonbullet3}\iff\;&\big(T(a_1\rightharpoonup b_1)\bullet (a_2\rightharpoonup b_2)\bullet (a_4\leftharpoonup b_4)\big)\ot \big((a_3\rightharpoonup b_3)\bullet (a_5\leftharpoonup b_5)\big)\\
 &= \big(T(a_1\rightharpoonup b_1)\bullet (a_2\rightharpoonup b_2)\bullet (a_4\leftharpoonup b_4)_1\big)\ot \big((a_3\rightharpoonup b_3)\bullet (a_4\leftharpoonup b_4)_2\big)\\
 \iff\;& (a_2\leftharpoonup b_2)\ot \big( (a_1\rightharpoonup b_1)\bullet (a_3\leftharpoonup b_3)\big) = (a_2\leftharpoonup b_2)_1\ot \big( (a_1\rightharpoonup b_1)\bullet (a_2\leftharpoonup b_2)_2\big)\\
 \iff\;& (a_3\leftharpoonup b_3)\ot  \big(T(a_1\rightharpoonup b_1)\bullet (a_2\rightharpoonup b_2)\bullet (a_4\leftharpoonup b_4)\big)\\
 & = (a_3\leftharpoonup b_3)_1\ot \big( T(a_1\rightharpoonup b_1)\bullet (a_2\rightharpoonup b_2)\bullet (a_3\leftharpoonup b_3)_2\big)\\
 \iff\;& (a_1\leftharpoonup b_1)\ot (a_2\leftharpoonup b_2) = (a\leftharpoonup b)_1\ot (a\leftharpoonup b)_2.
\end{align*} 
We prove that \eqref{compharpoonbullet2} and \eqref{compharpoonpbullet} are both equivalent to \eqref{compharpoonbullet3}. 
    One has 
\begin{align*}
    &\hspace{-2em}(a_2\leftharpoonup b_2)\ot \big((a_1\rightharpoonup b_1)\bullet (a_3\leftharpoonup b_3)\big) = (a_1\leftharpoonup b_1)\ot (a_2\bullet b_2)\\
    &\Leftrightarrow\big((a_1\rightharpoonup b_1)\bullet(a_3\leftharpoonup b_3)\big)\ot \big((a_2\rightharpoonup b_2)\bullet (a_4\leftharpoonup b_4)\big) = \big((a_1\rightharpoonup b_1)\bullet (a_2\leftharpoonup b_2)\big)\ot (a_3\bullet b_3)\\
    &\Leftrightarrow \big((a_1\rightharpoonup b_1)\bullet(a_3\leftharpoonup b_3)\big)\ot \big((a_2\rightharpoonup b_2)\bullet (a_4\leftharpoonup b_4)\big) = (a_1\bullet b_1)\ot (a_2\bullet b_2)\\
    &\Leftrightarrow\eqref{compharpoonbullet3}.
    \end{align*}
and
\begin{align*}
        \eqref{compharpoonbullet3}&\iff&&\!\!\big((a_{1}\rightharpoonup b_{1})\bullet(a_{3}\leftharpoonup b_{3})\big)\ot\big((a_{2}\rightharpoonup b_{2})\bullet (a_4\leftharpoonup b_4)\big)=(a_{1}\bullet b_{1})\ot(a_2\bullet b_2)\\
        &\iff &&\!\!\big((a_{1}\rightharpoonup b_{1})\bullet(a_{3}\leftharpoonup b_{3})\big)\ot\big((a_{2}\rightharpoonup b_{2})\bullet (a_4\leftharpoonup b_4)\big)=(a_{1}\bullet b_{1})\ot\big((a_{2}\rightharpoonup b_{2})\bullet (a_3\leftharpoonup b_3)\big)\\
        &\iff &&\!\!\big((a_{1}\rightharpoonup b_{1})\bullet(a_{3}\leftharpoonup b_{3})\big)\ot(a_{2}\rightharpoonup b_{2})=(a_{1}\bullet b_{1})\ot(a_{2}\rightharpoonup b_{2}),
    \end{align*}
where we use the fact that \eqref{compharpoonbullet3} implies \eqref{delta-colinear} and $\epsilon(a\leftharpoonup b)=\epsilon(a)\epsilon(b)$ in the latter equivalence. This completes the proof that \eqref{delta-colinear}--\eqref{compharpoonpbullet} are equivalent. Finally, using the assumption \eqref{braided-commutativity}, \eqref{matchedpair5} clearly implies \eqref{compharpoonbullet2}; and, conversely, assuming \eqref{compharpoonbullet2} (and thus \eqref{delta-colinear}) we immediately retrieve \eqref{matchedpair5}.
\end{proof}
\begin{corollary}
    Let $(H,\bullet,1,\Delta,\epsilon,T)$ be a Hopf algebra. Every matched pair of actions $(H,\rightharpoonup,\leftharpoonup)$ is a matched pair $(H,H)$ on $H$.
\end{corollary}

\begin{lemma}\label{rbraiding}
    Let $(H,\rightharpoonup,\leftharpoonup)$ be as in the setting of Proposition \ref{proposition:algebra}. Then, $(H,\rightharpoonup,\mathrm{Ad}_{L})$ is in ${}^H_H\mathcal{YD}$, where $\mathrm{Ad}_{L}:H\to H\otimes H$, $a\mapsto a_{1}\bullet T(a_{3})\otimes a_{2}$ is the left $H$-adjoint coaction. As a consequence, denoted by $\sigma^\mathcal{YD}$ the braiding of $^H_H\mathcal{YD}$, one has that $\sigma^{\mathcal{YD}}_{H,H}\colon H\otimes H\to H\otimes H$ is explicitly described by $\sigma^{\mathcal{YD}}_{H,H}\colon a\ot b\mapsto(a_{1}\bullet T(a_{3})\rightharpoonup b)\ot a_{2}$.
\end{lemma}
\begin{proof}
    We already know that $(H,\rightharpoonup)$ is a left $H$-module and $(H,\mathrm{Ad}_{L})$ is a left $H$-comodule. Since $\leftharpoonup$ is a morphism of coalgebras, the equivalent conditions of Lemma \ref{lemma:equiv-cond-colinearity} hold true. We check the compatibility of Yetter--Drinfeld modules \eqref{YD}:
    \begin{align*}
        \mathrm{Ad}_L(a\rightharpoonup b)& =\Big((a_{1}\rightharpoonup b_{1})\bullet T(a_{3}\rightharpoonup b_{3})\Big)\otimes(a_{2}\rightharpoonup b_{2})\\\overset{\eqref{T-on-harpoon}}&{=}(a_{1}\rightharpoonup b_{1})\bullet(a_{3}\leftharpoonup b_{3})\bullet T(b_{4})\bullet T(a_{4})\otimes(a_{2}\rightharpoonup b_{2})\\\overset{\eqref{compharpoonpbullet}}&{=}a_{1}\bullet b_{1}\bullet T(b_{3})\bullet T(a_{3})\otimes(a_{2}\rightharpoonup b_{2})\\&=a_{1}\bullet b_{-1}\bullet T(a_{3})\otimes(a_{2}\rightharpoonup b_{0}).
\end{align*}
\end{proof}

\begin{remark}
    Notice that the braiding $\sigma^{\mathcal{YD}}_{H,H}$ coincides with the flip map $\tau$ in case the adjoint coaction $\mathrm{Ad}_{L}$ is trivial, i.e. $a_{1}\bullet T(a_{3})\ot a_{2}=1\ot a$, which clearly happens in the cocommutative case. Conversely, if $a_{1}\bullet T(a_{3})\ot a_{2}=1\ot a$, then
\[
a_{1}\bullet T(a_{3})\ot a_{2}=1\ot a=1\ot a_{1}\epsilon(a_{2})=a_{2}\bullet T(a_{3})\ot a_{1}
\]
and so $a_{1}\ot a_{2}=a_{2}\ot a_{1}$, i.e. $H$ is cocommutative. Furthermore, let us also observe that $\sigma^{\mathcal{YD}}_{H,H}$ coincides with $\tau$ also in case $\rightharpoonup$ is the trivial action. 
\end{remark}
\begin{proposition}\label{prop:S-is-antipode}
    Let $H$ be a Hopf algebra, $(B, \cdot, 1_B, \Delta_B, \epsilon_B)$ a bialgebra in ${}^H_H\mathcal{YD}$. Suppose that $S\colon B\to B$ is linear map, and a convolution inverse of $\id_B$: then, $B$ is a Hopf algebra in $^H_H\mathcal{YD}$, with antipode $S$.
\end{proposition}
\begin{proof}\hspace{-4pt}\footnote{We are grateful to A.\@ Ardizzoni for suggesting this proof.}
    Define $\mathrm{can}\colon B\ot B \to B\ot B$ as $\mathrm{can}:=(m_{B}\ot m_{B})(\id_B\ot\sigma^{\mathcal{YD}}_{B,B}\ot\id_{B})(\id_{B}\ot u_{B}\ot\Delta_{B}) $ which is a composition of morphisms in ${}^H_H\mathcal{YD}$, and hence a morphism in ${}^H_H\mathcal{YD}$. The map $\mathrm{can}$ is explicitly given by $\mathrm{can}(a\ot b) = (a\cdot b_{1})\otimes b_{2} $. Since $S$ is a convolution inverse of $\mathrm{Id}_{B}$, $\mathrm{can}$ is invertible with inverse given by $\mathrm{can}^{-1}\colon a\ot b \mapsto (a\cdot S(b_1))\ot b_2$. Since $\mathrm{can}$ is invertible and it is a morphism in $^{H}_{H}\mathcal{YD}$, the inverse is also a morphism in ${}^H_H\mathcal{YD}$. Now observe that $S = (\id_B\ot \epsilon_B)\mathrm{can}^{-1}(u_B\ot \id_B)$, thus $S$ is a composition of morphisms in ${}^H_H\mathcal{YD}$: thus $S$ is a morphism in ${}^H_H\mathcal{YD}$, and hence $B$ is in $\mathrm{Hopf}({}^H_H\mathcal{YD})$.
\end{proof}
We merge the previous results in the following theorem:
\begin{theorem}\label{thm:matched-pair-yields-YDbrace}
    Let $(H,\bullet,1,\Delta,\epsilon,T)$ be a Hopf algebra, $(H,\rightharpoonup,\leftharpoonup)$ a matched pair of actions. Define $\cdot$ and $S$ as in \eqref{def-cdot},\eqref{def-S}. Then: 
    \begin{enumerate}
        \item[\textit{i.}]$(H,\cdot,1,\Delta,\epsilon,S)$ is in $\mathrm{Hopf}(^{H}_{H}\mathcal{YD})$ with the action $\rightharpoonup$ and the adjoint coaction $\mathrm{Ad}_L$;
        \item[\textit{ii.}]the two operations $\bullet$ and $\cdot$ satisfy the compatibility condition \eqref{compHopfbraces}.
    \end{enumerate}
\end{theorem}
\begin{proof}
The compatibility \eqref{compHopfbraces} is satisfied by Proposition \ref{proposition:algebra}.

We first check that $\cdot$ is a morphism of coalgebras in $^{H}_{H}\mathcal{YD}$. Using Lemma \ref{lem:deltadot} we obtain:
\[
\begin{split}
    \Delta(a\cdot b)&=(a_{1}\bullet(T(a_{3})\rightharpoonup b_{1}))\ot(a_{2}\cdot b_{2})\\\overset{\eqref{compblulletdotharp}}&{=}(a_{1}\cdot(a_{2}\rightharpoonup(T(a_{4})\rightharpoonup b_{1})))\ot(a_{3}\cdot b_{2})\\&=a_{1}\cdot(a_{2}\bullet T(a_{4})\rightharpoonup b_{1})\ot(a_{3}\cdot b_{2}),
\end{split}
\]
i.e., $\Delta m_{\cdot}=(m_{\cdot}\ot m_{\cdot})(\id\ot\sigma^{\mathcal{YD}}_{H,H}\ot\id)(\Delta\ot\Delta)$. By Proposition \ref{proposition:algebra} we also know that $\epsilon(a\cdot b)=\epsilon(a)\epsilon(b)$. Together, they are the compatibility conditions of a bialgebra in $^{H}_{H}\mathcal{YD}$.

Our next step is proving that $(H,\Delta, \epsilon)$ is a coalgebra in $^H_H\mathcal{YD}$. We know that $\Delta$ and $\epsilon$ are left $H$-linear with respect to $\rightharpoonup$. We thereby prove that they are also left $H$-colinear. We compute
\begin{align*}
(\id\ot\Delta)\mathrm{Ad}_{L}(a)&=(\id\ot\Delta)(a_{1}\bullet T(a_{3})\otimes a_{2})\\&=a_{1}\bullet T(a_{4})\otimes a_{2}\otimes a_{3}\\&=(a_{1}\bullet T(a_{3})\bullet a_{4}\bullet T(a_{6}))\otimes a_{2}\otimes a_{5}\\&=(a_{1_{1}}\bullet T(a_{1_{3}}))\bullet(a_{2_{1}}\bullet T(a_{2_{3}}))\otimes a_{1_{2}}\otimes a_{2_{2}},
\end{align*}
and
\[
(\id\ot\epsilon)\mathrm{Ad}_{L}(a)=(a_{1}\bullet T(a_{3}))\otimes\epsilon(a_{2})=(a_{1}\bullet T(a_{2}))\otimes1_{\Bbbk}=\epsilon(a)1_{H}\otimes1_{\Bbbk}.
\]
This concludes the proof that $(H,\Delta,\epsilon)$ is a coalgebra in $^{H}_{H}\mathcal{YD}$.

Clearly, the unit $u\colon 1_\Bbbk\mapsto 1_H$ is left $H$-linear and left $H$-colinear, so in order to conclude that $(H,\cdot,1)$ is in $\mathrm{Mon}(^{H}_{H}\mathcal{YD})$ we only need to verify that $\cdot$ is left $H$-linear and left $H$-colinear.

First we show that $\cdot$ is left $H$-linear with respect to $\rightharpoonup$. Therefore, we compute
\begin{align*}
a\rightharpoonup(b\cdot c)\overset{\eqref{def-cdot}}&{=}a\rightharpoonup(b_{1}\bullet(T(b_{2})\rightharpoonup c))\\\overset{\eqref{matchedpair-3}}&{=}(a_{1}\rightharpoonup b_{1})\bullet((a_{2}\leftharpoonup b_{2})\rightharpoonup(T(b_{3})\rightharpoonup c))\\\overset{\eqref{compblulletdotharp}}&{=}(a_{1}\rightharpoonup b_{1})\cdot((a_{2}\rightharpoonup b_{2})\rightharpoonup((a_{3}\leftharpoonup b_{3})\rightharpoonup(T(b_{4})\rightharpoonup c))\\&=(a_{1}\rightharpoonup b_{1})\cdot(((a_{2}\rightharpoonup b_{2})\bullet(a_{3}\leftharpoonup b_{3}))\rightharpoonup(T(b_{4})\rightharpoonup c))\\\overset{\eqref{braided-commutativity}}&{=}(a_{1}\rightharpoonup b_{1})\cdot(a_{2}\bullet b_{2}\rightharpoonup(T(b_{3})\rightharpoonup c))\\&=(a_{1}\rightharpoonup b_{1})\cdot(a_{2}\bullet b_{2}\bullet T(b_{3})\rightharpoonup c))\\&=(a_{1}\rightharpoonup b)\cdot(a_{2}\rightharpoonup c).
\end{align*}
We now check that $\cdot$ is left $H$-colinear with respect to $\mathrm{Ad}_L$. We already know that \eqref{triple-delta} holds, i.e. 
\[
(\Delta\ot\id)\Delta(a\cdot b)=(a_{1}\bullet(T(a_{6})\rightharpoonup b_{1}))\otimes(a_{2}\bullet(T(a_{5})\rightharpoonup b_{2}))\otimes(a_{3}\bullet(T(b_{4})\rightharpoonup b_{3})).
\]
The request that $\cdot$ is colinear with respect to the adjoint coaction $\mathrm{Ad}_L$ reads as follows:
\begin{equation}\label{eq:bullet-colinear}
(a\cdot b)_{1}\bullet T((a\cdot b)_{3})\ot(a\cdot b)_{2}=a_{1}\bullet T(a_{3})\bullet b_{1}\bullet T(b_{3})\ot(a_{2}\cdot b_{2}).
\end{equation}
We now rephrase the condition \eqref{eq:bullet-colinear}. One has
\begin{align*}
    &\hspace{-2em}(a\cdot b)_{1}\bullet T((a\cdot b)_{3})\ot(a\cdot b)_{2}\\
    &=\bigg(a_1\bullet (T(a_6)\rightharpoonup b_1)\bullet T\Big(a_3\bullet (T(a_4)\rightharpoonup b_3) \Big)\bigg)\ot \Big(a_2\bullet (T(a_5)\rightharpoonup b_2)\Big)\\
    &= \bigg(a_1\bullet (T(a_6)\rightharpoonup b_1)\bullet T\Big(T(a_4)\rightharpoonup b_3 \Big)\bullet T(a_3)\bigg)\ot \Big(a_2\bullet (T(a_5)\rightharpoonup b_2)\Big)\\
    &= \bigg(a_1\bullet (T(a_6)\rightharpoonup b_1)\bullet T\Big(T(a_4)\rightharpoonup b_3 \Big)\bullet T(a_3)\bullet b_4\bullet T(b_5)\bigg)\ot \Big(a_2\bullet (T(a_5)\rightharpoonup b_2)\Big)\\ \overset{\eqref{rightaction-from-leftaction}}&{=} \Big(a_1\bullet (T(a_5)\rightharpoonup b_1)\bullet (T(a_3)\leftharpoonup b_3)\bullet T(b_4)\Big)\ot\Big(a_2\bullet (T(a_4)\rightharpoonup b_2)\Big).
\end{align*}
Therefore, \eqref{eq:bullet-colinear} becomes
\begin{equation}\label{eq:50}
\begin{split}&\Big(a_1\bullet (T(a_5)\rightharpoonup b_1)\bullet (T(a_3)\leftharpoonup b_3)\bullet T(b_4)\Big)\ot\Big( a_2\bullet (T(a_4)\rightharpoonup b_2)\Big)\\
&\hspace{2em}= \Big(a_1\bullet T(a_4)\bullet b_1\bullet T(b_3)\Big)\ot \Big( a_2\bullet (T(a_3)\rightharpoonup b_2) \Big).
\end{split}
\end{equation}
But now, using the antipode, we obtain
\begin{align*}
    \eqref{eq:50}\iff\;&\Big((T(a_4)\rightharpoonup b_1)\bullet (T(a_2)\leftharpoonup b_3)\Big)\ot\big( a_1\bullet (T(a_3)\rightharpoonup b_2)\Big)\\
    &= \big(T(a_3)\bullet b_1\big)\ot \Big( a_1\bullet (T(a_2)\rightharpoonup b_2) \Big)\\\iff\;&\Big((T(a_3)\rightharpoonup b_1)\bullet (T(a_1)\leftharpoonup b_3)\Big)\ot\Big(  T(a_2)\rightharpoonup b_2\Big)= \big(T(a_2)\bullet b_1\big)\ot \big( T(a_1)\rightharpoonup b_2 \big).
\end{align*}
Thus, we obtain that \eqref{eq:bullet-colinear} is equivalent to
\begin{equation}
    \label{quasiequivalent-to-colinear} (T(a_3)\rightharpoonup b_1)\bullet (T(a_1)\leftharpoonup b_3)\ot (T(a_2)\rightharpoonup b_2) = \big(T(a_2)\bullet b_1\big)\ot \big( T(a_1)\rightharpoonup b_2\big).
\end{equation}
The latter is implied by \eqref{compharpoonpbullet}, which holds true since $\leftharpoonup$ is a morphism of coalgebras (and is in fact equivalent to \eqref{compharpoonpbullet} when $T$ is bijective). This concludes the proof that $(H,\cdot, 1)$ is an algebra in $^H_H\mathcal{YD}$, and hence $(H,\cdot, 1,\Delta, \epsilon)$ is a bialgebra in $^H_H\mathcal{YD}$ for what already observed.

We already know from Proposition \ref{proposition:algebra} that $S$ satisfies $S(a_1)\cdot a_2 = \epsilon(a)1_{H}= a_1\cdot S(a_2)$, thus we only need to prove that $S$ is a morphism of Yetter--Drinfeld modules. This follows from Proposition \ref{prop:S-is-antipode}.
\end{proof}
\begin{definition}
    Given a matched pair $(H, \rightharpoonup, \leftharpoonup)$ on a Hopf algebra $(H,\bullet, 1, \Delta, \epsilon, T)$, the second structure $H^{\raisebox{1pt}{.}}$ from Theorem \ref{thm:matched-pair-yields-YDbrace} \textit{i.\@} will be called the \emph{transmutation} of $H^\bullet$ with respect to the matched pair.
\end{definition}
In Remark \ref{rem:transmutation}, we shall see that this is indeed a generalisation of Majid's transmutation \cite{Maj2}.
\begin{remark}
    By \cite[Remark 1.8]{Pointed} the Hopf algebra $(H,\cdot,1,\Delta, \epsilon, S)$ is, in particular, a braided Hopf algebra in the sense of Kharchenko \cite[\S$\,$2.2]{Kha}. If moreover $T$ is bijective, then the braiding $\sigma^\mathcal{YD}$ is bijective, hence $(H,\cdot,1,\Delta, \epsilon, S)$ is a braided Hopf algebra in the sense of Majid \cite{majid1995algebras}; see also \cite[Definition 5.1]{BraidedHopfalgebras}.
\end{remark}
\begin{remark}
    In case $H$ is cocommutative, $\sigma^{\mathcal{YD}}_{H,H}$ coincides with the flip $\tau$, hence $(H,\cdot,1,\Delta,\epsilon,S)$ becomes a standard Hopf algebra and we recover the definition of Hopf brace \cite{angiono2017hopf}. 
\end{remark}
%
%
\begin{corollary}\label{cor:bosonisation}
    Let $(H,\bullet,1,\Delta,\epsilon,T)$ be a Hopf algebra and $(H,\rightharpoonup,\leftharpoonup)$ a matched pair of actions. Define $\cdot$ as in \eqref{def-cdot}, and $S$ as in \eqref{def-S}. Then, the following is a bialgebra structure on $H\ot H$:
\begin{align*}
        &(a\ot h) \cdot_{\#} (a'\ot h') := a\cdot S(h_{1})\cdot (h_{2}\bullet a')\ot h_3\bullet h',\\
        &\Delta_{\#}(a\ot h):= a_1\ot a_2\bullet T(a_4)\bullet h_1\ot a_3\ot h_2,\\
&S_{\#}(a\ot h):=S\big(T(T(a_{3})\bullet h_{3})\big)\cdot\big(T(T(a_{4})\bullet h_{2})\bullet T(a_{2})\big)\ot T(a_{1}\bullet T(a_{5})\bullet h_{1}).\end{align*}
\end{corollary}
\begin{proof}
    Let $H^\bullet:=(H,\bullet)$ and $H^{\raisebox{1pt}{.}}:=(H,\cdot)$ denote the two multiplicative structures. By Theorem \ref{thm:matched-pair-yields-YDbrace} we have that $(H,\cdot,1,\Delta,\epsilon)$ is in $\mathrm{Bimon}({}^{H}_{H}\mathcal{YD})$, thus the bosonisation $(H^{\raisebox{1pt}{.}}\#H^{\bullet},\cdot_{\#},1\otimes1,\Delta_{\#},\epsilon_{H}\ot\epsilon_{H})$ is a bialgebra, where $(a\ot h)\cdot_{\#}(a'\ot h'):=a\cdot(h_{1}\rightharpoonup a')\ot h_{2}\bullet h'$ and $\Delta_{\#}(a\ot h):=a_{1}\ot a_{2_{-1}}\bullet h_{1}\ot a_{2_{0}}\ot h_{2}$ denote the smash product and the smash coproduct, respectively, see e.g. Heckenberger and Schneider \cite[Proposition 3.8.4]{HecSch}. Moreover, since $H^{\bullet}$ and $H^{\raisebox{1pt}{.}}$ are both Hopf algebras, we obtain an antipode on $H^{\raisebox{1pt}{.}}\#H^{\bullet}$ given by $S_{\#}(a\ot h):=\big(T(a_{-1}\bullet h_{2})\rightharpoonup S(a_{0})\big)\ot T(a_{-2}\bullet h_{1})$; see e.g.\@ \cite[Theorem 3.8.10]{HecSch}. Thus, we compute
\begin{align*}
    S_{\#}(a\ot h)&=\big(T(a_{-1}\bullet h_{2})\rightharpoonup S(a_{0})\big)\ot T(a_{-2}\bullet h_{1})\\&=\big(T(a_{2}\bullet T(a_{4})\bullet h_{2})\rightharpoonup S(a_{3})\big)\ot T(a_{1}\bullet T(a_{5})\bullet h_{1})\\&=\big(T(T(a_{5})\bullet h_2)\bullet T(a_{2})\rightharpoonup(a_{3}\rightharpoonup T(a_{4}))\big)\ot T(a_{1}\bullet T(a_{6})\bullet h_{1})\\&=\big(T(T(a_{3})\bullet h_{2})\rightharpoonup T(a_{2})\big)\ot T(a_{1}\bullet T(a_{4})\bullet h_{1})\\&=S\big(T(T(a_{3})\bullet h_{3})\big)\cdot\big(T(T(a_{4})\bullet h_{2})\bullet T(a_{2})\big)\ot T(a_{1}\bullet T(a_{5})\bullet h_{1}).
\end{align*}
\end{proof}

Theorem \ref{thm:matched-pair-yields-YDbrace} suggests us the following definition:
\begin{definition}\label{def:YDbrace}
    A \emph{Yetter--Drinfeld brace} (or \emph{$\mathcal{YD}$-brace}) $(H, \cdot, \bullet,1,\Delta,\epsilon,S,T)$ is the datum of a Hopf algebra $H^\bullet = (H,\bullet, 1,\Delta, \epsilon, T)$, a second operation $\cdot$ on $H$, and a linear map $S\colon H\to H$ such that:
    \begin{enumerate}
        \item[\textit{i.}]$(H,\cdot,1,\Delta,\epsilon,S)$ is in $\mathrm{Hopf}({}^{H^\bullet}_{H^\bullet}\mathcal{YD})$ with the action $\rightharpoonup$ defined by $a\rightharpoonup b := S(a_1)\cdot (a_2\bullet b)$, and the coaction given by $\mathrm{Ad}_L$; 
        \item[\textit{ii.}] if we define $a\leftharpoonup b:= T(a_1\rightharpoonup b_1)\bullet a_2\bullet b_2$, the two maps $\rightharpoonup, \leftharpoonup$ satisfy \eqref{matchedpair5};
        \item[\textit{iii.}]the two operations $\bullet$ and $\cdot$ satisfy the Hopf brace compatibility \eqref{compHopfbraces}.
         \end{enumerate}
         Given two Yetter--Drinfeld braces $H$ and $K$, a \emph{morphism of Yetter--Drinfeld braces} is a map $f\colon H\to K$ that is a morphism in $\mathrm{Hopf}(\mathsf{Vec}_\Bbbk)$ between the respective Hopf algebras in $\mathsf{Vec}_\Bbbk$, and satisfies
         \[f(a\cdot_H b) = f(a)\cdot_K f(b).\]In particular, this implies $S_{K}f=fS_{H}$, $f\rightharpoonup_{H}\;=\;\rightharpoonup_{K}\!(f\otimes f)$ and $(\mathrm{Ad}_{L})_Kf=(f\ot f)(\mathrm{Ad}_{L})_H$. We denote the category of Yetter--Drinfeld braces by $\mathcal{YD}\mathrm{Br}(\mathsf{Vec}_\Bbbk)$.
\end{definition}
\begin{remark}
    Let $f\colon H\to K$ be a morphism of Yetter--Drinfeld braces. Since $f\colon H^{\bullet}\to K^{\bullet}$ is a morphism of bialgebras, one can regard $K^{\raisebox{1pt}{.}}$ as a left $H^{\bullet}$-module through $f$, so that the condition $f(a\rightharpoonup_{H} b)=f(a)\rightharpoonup_{K}f(b)$ means that $f$ is a morphism in $_{H^{\bullet}}\mm$; and $H^{\raisebox{1pt}{.}}$ as a left $K^{\bullet}$-comodule through $f$, so that the condition $(\mathrm{Ad}_{L})_{K}f=(f\otimes f)(\mathrm{Ad}_{L})_{H}$ means that $f$ is a morphism in $^{K^{\bullet}}\mm$. One can consider the category $^{K^{\bullet}}_{H^{\bullet}}\mathcal{YD}$ of \textit{relative} Yetter--Drinfeld modules; see \cite[\S$\,$4.4, p.\@ 183]{CSP}. Clearly $H^{\raisebox{1pt}{.}}$ and $K^{\raisebox{1pt}{.}}$ are in $^{K^{\bullet}}_{H^{\bullet}}\mathcal{YD}$. Indeed, given $a,b\in H$, we compute:
\begin{align*}
(a_{1}\rightharpoonup_{H} b)^K_{-1}\bullet_{K} f(a_{2})\ot(a_{1}\rightharpoonup_{H} b)^{K}_{0}&=f((a_{1}\rightharpoonup_{H}b)_{-1}\bullet_{H}a_{2})\otimes(a_{1}\rightharpoonup_{H}b)_{0}\\&=f(a_{1}\bullet_{H} b_{-1})\otimes(a_{2}\rightharpoonup_{H}b_{0})\\&=f(a_{1})\bullet_{K} b_{-1}^{K}\otimes(a_{2}\rightharpoonup_{H} b_{0}^K),
\end{align*}
where $a_{-1}\ot a_0 = (\mathrm{Ad}_L)_H(a)$ is the usual adjoint $H^\bullet$-coaction on $H$, while $a^K_{-1}\ot a^K_0= (f\ot \id)(\mathrm{Ad}_L)_H(a)$ denotes the $K^\bullet$-coaction on $H^{\raisebox{1pt}{.}}$ induced by $\mathrm{Ad}_L$ through $f$. The Yetter--Drinfeld compatibility for $K^{\raisebox{1pt}{.}}$ is deduced analogously. Hence we obtain that $f$ is a morphism in $^{K^{\bullet}}_{H^{\bullet}}\mathcal{YD}$. Moreover, one clearly has that $H^{\raisebox{1pt}{.}}$ and $K^{\raisebox{1pt}{.}}$ are in $\mathrm{Mon}(^{K^{\bullet}}_{H^{\bullet}}\mathcal{YD})$ and $\mathrm{Comon}(^{K^{\bullet}}_{H^{\bullet}}\mathcal{YD})$; thus the condition $f(a\cdot_{H}b)=f(a)\cdot_{K}f(b)$ means that $f:H^{\raisebox{1pt}{.}}\to K^{\raisebox{1pt}{.}}$ is a morphism in $\mathrm{Mon}(^{K^{\bullet}}_{H^{\bullet}}\mathcal{YD})$. Notice that $f $ is automatically a morphism in $\mathrm{Comon}(^{K^{\bullet}}_{H^{\bullet}}\mathcal{YD})$.
\end{remark}

From Theorem \ref{thm:matched-pair-yields-YDbrace}, we obtain:
\begin{corollary}\label{cor:Hyetterdrinfeldbrace}
    Let $(H,\bullet,1,\Delta,\epsilon,T)$ be a Hopf algebra, $(H,\rightharpoonup,\leftharpoonup)$ a matched pair of actions. Define $\cdot$ and $S$ as in \eqref{def-cdot}, \eqref{def-S}. Then, $(H,\cdot,\bullet,1,\Delta,\epsilon,S,T)$ is a Yetter--Drinfeld brace.
\end{corollary}
\begin{remark}\label{rem:leftharpcoalg} Define $\leftharpoonup$ as in Definition \ref{def:YDbrace} \emph{ii}. Since $\rightharpoonup$ is a morphism of coalgebras, it is clear that $\rightharpoonup, \leftharpoonup$ satisfy \eqref{braided-commutativity}:
\[
(a_1\rightharpoonup b_1)\bullet (a_2\leftharpoonup b_2)=(a_1\rightharpoonup b_1)\bullet T(a_{2}\rightharpoonup b_{2})\bullet a_{3}\bullet b_{3}=\epsilon(a_{1}\rightharpoonup b_{1})a_{2}\bullet b_{2}=a\bullet b.
\]
Observe that $\epsilon (a\leftharpoonup b) = \epsilon(a)\epsilon(b)$, thus we are in the hypotheses of Lemma \ref{lemma:equiv-cond-colinearity}. Consequently, Definition \ref{def:YDbrace} \emph{ii.\@} is equivalent to the request that $\leftharpoonup$ is a morphism of coalgebras.
\end{remark}
\begin{remark}
    Notice that Definition \ref{def:YDbrace} makes as much sense when we replace $\mathsf{Vec}_\Bbbk$ with a braided monoidal category. Although this exceeds the scope of this work, we may expect several of our results in $\mathsf{Vec}_\Bbbk$ to hold more in general. Furthermore, Definition \ref{def:YDbrace} can be dualised to define \emph{Yetter--Drinfeld cobraces}, leading to a generalisation of commutative Hopf cobraces \cite{angiono2017hopf}. 
\end{remark}
Clearly, in a Yetter--Drinfeld brace, the map $\rightharpoonup$ defined as in Definition \ref{def:YDbrace} \emph{i.\@} is a left $H^{\bullet}$-action, and $(H,\cdot,1)$ is a left $H^{\bullet}$-module algebra. In fact, this is true under weaker assumptions:
\begin{lemma}\label{lem:rightharpoon}
    Let $H^\bullet=(H,\bullet,1,\Delta,\epsilon)$ be a bialgebra and $(H,\cdot,1,\Delta,\epsilon)$ be an algebra and a coalgebra, endowed with a map $S\colon H\to H$ satisfying $S(a_{1})\cdot a_{2}=\epsilon(a)1=a_{1}\cdot S(a_{2})$. Suppose that the Hopf brace compatibility \eqref{compHopfbraces} between $\bullet $ and $\cdot$ is satisfied. Then
\[
a\rightharpoonup b:=S(a_{1})\cdot(a_{2}\bullet b)
\]
is a left $H^{\bullet}$-action on $H$. Moreover, $(H,\cdot,1)$ is a left $H^{\bullet}$-module algebra and, if $\epsilon$ is a morphism of algebras with respect to $\cdot$, then $\epsilon$ is left $H^{\bullet}$-linear with respect to $\rightharpoonup$.
\end{lemma}
\begin{proof}
    First we compute $1\rightharpoonup b=S(1)\cdot(1\bullet b)=b$. Moreover, we have
\begin{align*}
    a\rightharpoonup(b\rightharpoonup c)&=a\rightharpoonup(S(b_{1})\cdot(b_{2}\bullet c))\\&=S(a_{1})\cdot\Big(a_{2}\bullet\big(S(b_{1})\cdot(b_{2}\bullet c)\big)\Big)\\&=S(a_{1})\cdot(a_{2}\bullet S(b_{1}))\cdot S(a_{3})\cdot(a_{4}\bullet b_{2}\bullet c)\\\overset{(\dagger)}&{=}S(a_{1}\bullet b_{1})\cdot a_{2}\cdot S(a_{3})\cdot(a_{4}\bullet b_{2}\bullet c)\\&=S(a_{1}\bullet b_{1})\cdot (a_{2}\bullet b_{2}\bullet c)\\&=(a\bullet b)\rightharpoonup c,
\end{align*}
where the equality marked with $(\dagger)$ follows from the relation $S(a_{1}\bullet b)\cdot a_{2}=S(a_{1})\cdot(a_{2}\bullet S(b))$, whose proof is the same as in \cite[Lemma 1.7]{angiono2017hopf} and simply follows from \eqref{compHopfbraces}. Moreover, we have
\[
a\rightharpoonup(b\cdot c)=S(a_{1})\cdot(a_{2}\bullet(b\cdot c))=S(a_{1})\cdot(a_{2}\bullet b)\cdot S(a_{3})\cdot (a_{4}\bullet c)=(a_{1}\rightharpoonup b)\cdot(a_{2}\rightharpoonup c)
\]
and \begin{equation}\label{eq:matchedpair3}a\rightharpoonup1=S(a_{1})\cdot(a_{2}\bullet1)=S(a_{1})\cdot a_{2}=\epsilon(a)1,\end{equation} 
hence $(H,\cdot,1)$ is a left $H^{\bullet}$-module algebra. Finally, we have
\[
\epsilon(a\rightharpoonup b)=\epsilon(S(a_{1})\cdot(a_{2}\bullet b))\overset{(\ddagger)}{=}\epsilon(a_{1})\epsilon(a_{2})\epsilon(b)=\epsilon(a)\epsilon(b),
\]
where $(\ddagger)$ follows from the fact that $\epsilon$ is a morphism of algebras with respect to $\cdot$.
\end{proof}

\subsection{A general example} We now construct an easy example of a Yetter--Drinfeld brace on a Hopf algebra $H$. This construction depends on the assumption that the adjoint action on $H$ is compatible with the comultiplication.  
\begin{lemma}\label{matched-pairs-from-adjoint-action}
    Let $H^{\bullet}:=(H,\bullet,1,\Delta,\epsilon,T)$ be a Hopf algebra, $\rightharpoonup$ the left adjoint action $a\rightharpoonup b:=a_{1}\bullet b\bullet T(a_{2})$, and $\leftharpoonup$ the trivial action. Then, $(H,\rightharpoonup,\leftharpoonup)$ is a matched pair of actions if and only if $\Delta$ is left linear with respect to $\rightharpoonup$.
\end{lemma}
\begin{proof}
Notice that $H^{\bullet}$ is automatically a right $H$-module coalgebra. We verify \eqref{matchedpair-1}--\eqref{matchedpair-4} and \eqref{braided-commutativity}. Clearly $a\rightharpoonup1=\epsilon(a)1$ holds, while \eqref{matchedpair-2} is automatically true. Moreover, we compute
\[
(a_{1}\rightharpoonup b_{1})\bullet((a_{2}\leftharpoonup b_{2})\rightharpoonup c)=a_{1}\bullet b\bullet T(a_{2})\bullet(a_{3}\rightharpoonup c)=a_{1}\bullet b\bullet T(a_{2})\bullet a_{3}\bullet c\bullet T(a_{4})=a\rightharpoonup(b\bullet c),
\]
and
\[
(a\leftharpoonup(b_{1}\rightharpoonup c_{1}))\bullet(b_{2}\leftharpoonup c_{2})=a\,\epsilon(b_{1}\rightharpoonup c_{1})\bullet b_{2}\,\epsilon(c_{2})=(a\bullet b)\leftharpoonup c.
\]
Finally, 
\[
(a_{1}\rightharpoonup b_{1})\bullet(a_{2}\leftharpoonup b_{2})=a_{1}\bullet b_{1}\bullet T(a_{2})\bullet a_{3}\epsilon(b_{2})=a\bullet b.
\]
Thus, $(H,\rightharpoonup,\leftharpoonup)$ is a matched pair of actions if and only if $H^{\bullet}$ is a left $H$-module coalgebra. Since $\epsilon(a\rightharpoonup b)=\epsilon(a)\epsilon(b)$ is automatically satisfied, $H^\bullet$ is a left $H$-module coalgebra if and only if $\Delta$ is left linear with respect to $\rightharpoonup$; as desired.
\end{proof}

\begin{remark}
The linearity condition of $\Delta$ with respect to $\rightharpoonup$, i.e. $\Delta(a\rightharpoonup b)=(a_{1}\rightharpoonup b_{1})\otimes(a_{2}\rightharpoonup b_{2})$, explicitly reads:
\[
a_{1}\bullet b_{1}\bullet T(a_{4})\ot a_{2}\bullet b_{2}\bullet T(a_{3})=a_{1}\bullet b_{1}\bullet T(a_{2})\otimes a_{3}\bullet b_{2}\bullet T(a_{4}).
\]
The latter is equivalent to $T(a_{3})\ot a_{1}\bullet b\bullet T(a_{2})=T(a_{1})\ot a_{2}\bullet b\bullet T(a_{3})$, and then to 
\begin{equation}\label{compAdtrivialMP}
a_{1}\bullet T(a_{4})\ot a_{2}\bullet b\bullet T(a_{3})=1\ot a_{1}\bullet b\bullet T(a_{2}), 
\end{equation}
i.e., $(\id\ot\rightharpoonup)(\mathrm{Ad}_{L}\ot\id)=(u_{H}\otimes\mathrm{\rightharpoonup})$, where $\mathrm{Ad}_{L}$ denotes again the left adjoint coaction on $H$. From the matched pair of actions $(H,\rightharpoonup,\leftharpoonup)$ of Lemma \ref{matched-pairs-from-adjoint-action}, we can define $\cdot$ as in \eqref{def-cdot}, and $S$ as in \eqref{def-S}, obtaining
\begin{equation}\label{cdotSadj}
a\cdot b=a_{1}\bullet T(a_{3})\bullet b\bullet T(T(a_{2})),\qquad S(a)=a_{1}\bullet T(a_{3})\bullet T(a_{2}).
\end{equation}
Notice that \eqref{compAdtrivialMP} clearly holds in the commutative and cocommutative cases (where the latter was already considered, e.g., in Masuoka \cite[Remark p.\@ 302]{Masuoka}). If $H^{\bullet}$ is commutative, then $(H,\bullet,\bullet)$ is an example of a Yetter--Drinfeld brace. On the other hand, if $H^{\bullet}$ is cocommutative, then the previous construction reduces to the \emph{almost trivial} Hopf brace $(H,\bullet^{\mathrm{op}},\bullet)$ (the terminology is adapted from \cite{guarnieri2017skew}).
\end{remark}

In {\S}$\,$\ref{section:coquasi}, we shall see other explicit examples coming from coquasitriangular structures. In particular, the first example will be the Sweedler's Hopf algebra $H_4$. Generalisations of $H_4$ are given by the Taft algebras $H_{n^{2}}$, which do not admit any (co)quasitriangular structure for $n>2$ (see Gelaki \cite{Ge}). Matched pairs $(H_{n^{2}},H_{n^{2}})$ on $H_{n^{2}}$ are classified by Agore \cite[Theorem 2.1]{AgoreTaft}, and 
it is not difficult to see that none of these matched pairs satisfy \eqref{braided-commutativity}. This explicitly shows that matched pairs of actions $(H,\rightharpoonup,\leftharpoonup)$ correspond to a proper subclass of the matched pairs $(H,H)$ on $H$.
\subsection{A converse connection}\label{subsection:converse}
In order to show that our definition of Yetter--Drinfeld braces is optimal, we now prove that the correspondence of Theorem \ref{thm:matched-pair-yields-YDbrace} has a converse. Therefore, a Yetter--Drinfeld brace and a matched pair of actions are essentially the same thing. This is in fact, as we shall point out, an isomorphism of categories.

\begin{theorem}\label{thm:converse}
Let $(H,\cdot, \bullet, 1, \Delta, \epsilon, S, T)$ be a Yetter--Drinfeld brace. Define 
\[
a\leftharpoonup b:=T(a_{1}\rightharpoonup b_{1})\bullet a_{2}\bullet b_{2}.
\]
Then, $(H,\rightharpoonup,\leftharpoonup)$ is a matched pair of actions. 
\end{theorem}
\begin{proof}
Since $H$ is in $\mathrm{Comon}({}^{H^\bullet}_{H^\bullet }\mathcal{YD})$, one has that $\Delta$ and $\epsilon$ are left $H^\bullet$-linear with respect to $\rightharpoonup$, and hence $H$ is a left $H^\bullet$-module coalgebra. From the definition of $\rightharpoonup$ we immediately have $a\bullet b = a_1\cdot S(a_2)\cdot (a_3\bullet b) = a_1\cdot (a_2\rightharpoonup b)$, whence also
\[
    a\cdot b= a_1\cdot (a_2\rightharpoonup(T(a_3)\rightharpoonup b))=a_1\bullet (T(a_2)\rightharpoonup b).
\]We therefore compute
\begin{align*}
    a\leftharpoonup(b\bullet c)&=T(a_{1}\rightharpoonup(b_{1}\bullet c_{1}))\bullet a_{2}\bullet(b_{2}\bullet c_{2})\\ \nonumber &=T(S(a_{1})\cdot(a_{2}\bullet b_{1}\bullet c_{1}))\bullet a_{3}\bullet b_{2}\bullet c_{2}\\ \nonumber  &=T(S(a_{1})\cdot(a_{2}\bullet b_{1})\cdot S(a_{3}\bullet b_{2})\cdot(a_{4}\bullet b_{3}\bullet c_{1}))\bullet a_{5}\bullet b_{4}\bullet c_{2}\\ \nonumber &=T\big((a_{1}\rightharpoonup b_{1})\cdot((a_{2}\bullet b_{2})\rightharpoonup c_{1})\big)\bullet a_{3}\bullet b_{3}\bullet c_{2}\\ \nonumber &=T\Big((a_{1}\rightharpoonup b_{1})\bullet\big(T(a_{2}\rightharpoonup b_{2})\rightharpoonup((a_{3}\bullet b_{3})\rightharpoonup c_{1})\big)\Big)\bullet a_{4}\bullet b_{4}\bullet c_{2}\\ \nonumber &= T\Big(\big(T(a_{2}\rightharpoonup b_{2})\bullet a_{3}\bullet b_{3}\big)\rightharpoonup c_{1}\Big)\bullet T(a_{1}\rightharpoonup b_{1})\bullet a_{4}\bullet b_{4}\bullet c_{2}\\ \nonumber &=(T(a_{1}\rightharpoonup b_{1})\bullet a_{2}\bullet b_{2})\leftharpoonup c\\ \nonumber &=(a\leftharpoonup b)\leftharpoonup c.
\end{align*}
Moreover, we have $a\leftharpoonup1=T(a_{1}\rightharpoonup1)\bullet a_{2}\bullet1=T(\epsilon(a_{1})1)\bullet a_{2}=a$, thus $\leftharpoonup$ is a right $H^{\bullet}$-action. By Remark \ref{rem:leftharpcoalg} we already know that $\Delta$ is right $H^{\bullet}$-linear with respect to $\leftharpoonup$, i.e.  $\Delta(a\leftharpoonup b) = (a_1\leftharpoonup b_1)\ot (a_2\leftharpoonup b_2)$; that $H$ is a right $H^\bullet$-module coalgebra, and that \eqref{braided-commutativity} is satisfied.
We already know that \eqref{matchedpair-1} is satisfied. Moreover, \eqref{matchedpair-2} also holds:
\begin{equation}\label{eq:matchedpair4}1\leftharpoonup a=T(1\rightharpoonup a_{1})\bullet1\bullet a_{2}=T(a_{1})\bullet a_{2}=\epsilon(a)1.
\end{equation}
In order to show that $(H,\rightharpoonup,\leftharpoonup)$ is a matched pair of actions on $H$, it remains to prove that \eqref{matchedpair-3} and \eqref{matchedpair-4} are satisfied. We thereby compute
\begin{align*}
    (a_1\rightharpoonup b_1)\bullet \Big( \big( a_2\leftharpoonup b_2\big)\rightharpoonup c\Big)&= (a_1\rightharpoonup b_1)\bullet \Big( \big( T(a_2\rightharpoonup b_2)\bullet a_3\bullet b_3\big)\rightharpoonup c\Big)\\ \nonumber
    &= (a_1\rightharpoonup b_1)\bullet \Big(T(a_2\rightharpoonup b_2)\rightharpoonup \big(a_3\bullet b_3\rightharpoonup c\big)\Big)\\ \nonumber
    &= (a_1\rightharpoonup b_1)\cdot \big(a_2\bullet b_2\rightharpoonup c\big)\\ \nonumber
    &= (a_1\rightharpoonup b_1)\cdot (a_2\rightharpoonup(b_2\rightharpoonup c))\\ \nonumber
    \overset{(\dagger)}&{=} a\rightharpoonup (b_1\cdot (b_2\rightharpoonup c))\\ \nonumber
    &= a\rightharpoonup (b\bullet c),
\end{align*}
where $(\dagger)$ follows since $(H,\cdot,1)$ is a left $H^{\bullet}$-module algebra. Finally, one has
\begin{align*}
    (a\leftharpoonup(b_1\rightharpoonup c_1))\bullet (b_2\leftharpoonup c_2)&= T(a_1\rightharpoonup (b_1\rightharpoonup c_1))\bullet a_2\bullet (b_2\rightharpoonup c_2)\bullet (b_3\leftharpoonup c_3)\\ \nonumber
    \overset{\eqref{braided-commutativity}}&{=} T(a_1\rightharpoonup (b_1\rightharpoonup c_1))\bullet a_2\bullet b_2\bullet c_2\\ \nonumber
    &= T(a_1\bullet b_1\rightharpoonup c_1)\bullet a_2\bullet b_2\bullet c_2\\ \nonumber
    &= (a\bullet b)\leftharpoonup c.
\end{align*}
\end{proof}

The previous construction is converse to the one of Theorem \ref{thm:matched-pair-yields-YDbrace}.
\begin{theorem}
    The correspondence established in Theorem \ref{thm:matched-pair-yields-YDbrace} and Theorem \ref{thm:converse} is an isomorphism between $\mathcal{YD}\mathrm{Br}(\mathsf{Vec}_\Bbbk)$ and $\mathrm{MP}(\mathsf{Vec}_\Bbbk)$.
\end{theorem}
\begin{proof} 
Since the constructions of Theorem \ref{thm:matched-pair-yields-YDbrace} and of Theorem \ref{thm:converse} are converse to each other, if we prove that these constructions are functorial then they automatically yield an isomorphism of categories.

Let $F\colon \mathrm{MP}(\mathsf{Vec}_\Bbbk)\to \mathcal{YD}\mathrm{Br}(\mathsf{Vec}_\Bbbk)$ be the assignment defined in Theorem \ref{thm:matched-pair-yields-YDbrace}. We define the assignment on morphisms as follows: if $f\colon (H,\rightharpoonup_H,\leftharpoonup_H)\to (K,\rightharpoonup_K,\leftharpoonup_K)$ is a morphism in $\mathrm{MP}(\mathsf{Vec}_\Bbbk)$, the corresponding morphism in $\mathcal{YD}\mathrm{Br}(\mathsf{Vec}_\Bbbk)$ is $f\colon H\to K$ read as a morphism of Hopf algebras. 
Since $m_{\cdot_{H}} = m_{\bullet_{H}}(\id_{H}\ot \rightharpoonup_{H})(\id_{H}\ot T_{H}\ot \id_{H})(\Delta_{H}\ot \id_{H})$ (and a similar formula holds for $m_{\cdot_{K}}$), one has 
\begin{align*}
    m_{\cdot_{K}}(f\otimes f) &=m_{\bullet_K}(\id_{K}\ot\rightharpoonup_{K})(\id_{K}\ot T_{K}\ot\id_{K})(\Delta_{K}\ot\id_{K})(f\ot f)\\&=m_{\bullet_K}(\id_{K}\ot\rightharpoonup_{K})(\id_{K}\ot T_{K}\ot\id_{K})(f\ot f\ot f)(\Delta_{H}\ot\id_{H})\\&=m_{\bullet_K}(\id_{K}\ot\rightharpoonup_{K})(f\ot f\ot f)(\id_{H}\ot T_{H}\ot\id_{H})(\Delta_{H}\ot\id_{H})\\&=m_{\bullet_K}(f\ot f)(\id_{H}\ot\rightharpoonup_{H})(\id_{H}\ot T_{H}\ot\id_{H})(\Delta_{H}\ot\id_{H})\\&=f m_{\bullet_H}(\id_{H}\ot\rightharpoonup_{H})(\id_{H}\ot T_{H}\ot\id_{H})(\Delta_{H}\ot\id_{H})\\&= f m_{\cdot_{H}} ,
\end{align*}
thus $f$ is indeed a morphism in $\mathcal{YD}\mathrm{Br}(\mathsf{Vec}_\Bbbk)$. Moreover, it is clear that $F$ sends the identity morphism into the identity morphism. Finally, let $f\colon (H,\rightharpoonup_H,\leftharpoonup_H)\to (K,\rightharpoonup_K,\leftharpoonup_K)$ and $g\colon (K,\rightharpoonup_K,\leftharpoonup_K)\to (L,\rightharpoonup_L,\leftharpoonup_L) $ be morphisms in $\mathrm{MP}(\mathsf{Vec}_\Bbbk)$, then $F$ sends the composition $gf$ in $\mathrm{MP}(\mathsf{Vec}_\Bbbk)$ into the set-theoretic composition of maps $gf$, which is the composition of morphisms in $\mathcal{YD}\mathrm{Br}(\mathsf{Vec}_\Bbbk)$.

Analogously, one proves that the assignment $G\colon \mathcal{YD}\mathrm{Br}(\mathsf{Vec}_\Bbbk)\to \mathrm{MP}(\mathsf{Vec}_\Bbbk)$ defined in Theorem \ref{thm:converse} is a functor.
\end{proof}

\section{Yetter--Drinfeld braces as 1-cocycles}\label{section:1-cocycles}\noindent It is known (see Angiono, Galindo and Vendramin \cite[Theorem 1.12]{angiono2017hopf}, and  Fernández Vilaboa and Gonzáles Rodríguez \cite[Theorem 2.7]{eq1cocycleHopfbrace}) that a Hopf brace in a braided monoidal category $\Mm$ is tantamount to a 1-cocycle of bialgebras in $\Mm$. We are going to recover an analogous characterisation for Yetter--Drinfeld braces. 
\begin{definition}
    Let $(H,\bullet,1,\Delta,\epsilon,T)$ be a Hopf algebra, $(A,\cdot_{A},1,\Delta,\epsilon)$ be a bialgebra in $^H_H\mathcal{YD}$, and let the action of $H$ on $A$ be denoted by $\rightharpoonup_{A}$. A \emph{Yetter--Drinfeld 1-cocycle} is an isomorphism of coalgebras $\pi\colon H\to A$ satisfying the \emph{1-cocycle condition}    \begin{equation}
        \label{eq:1-cocycle}\tag{\scshape 1c} \pi(a\bullet b) = \pi (a_1)\cdot_{A}(a_2\rightharpoonup_{A} \pi( b)) 
    \end{equation}
and satisfying the following conditions for all $a,b\in H$:
\begin{align}
    &(a_{1}\rightharpoonup_{A}\pi T(a_{2}))\cdot_{A}\pi(a_{3})=\epsilon(a)1_{A};\label{comppiT}\\
    \begin{split}&T\pi^{-1}(a_2 \rightharpoonup_A \pi(b_2))\bullet a_3\bullet b_3\ot \pi^{-1}(a_1\rightharpoonup_A \pi(b_1 ))\\
    &\hspace{2em}= T\pi^{-1}(a_1 \rightharpoonup_A \pi(b_1))\bullet a_2\bullet b_2\ot \pi^{-1}(a_3\rightharpoonup_A \pi(b_3 )).\end{split}\label{compleftharpoondelta}
\end{align}
Given two Yetter--Drinfeld 1-cocycles $\pi \colon H\to A $ and $\varpi\colon K\to B$, a \emph{morphism of Yetter--Drinfeld 1-cocycles} $\pi\to\varpi$ is a pair $(f,g)$, where $f\colon H\to K$ is a morphism of Hopf algebras, and $g\colon A\to B$ is a morphism of algebras and coalgebras satisfying
\[
g\rightharpoonup_{A}\;\,=\;\,\rightharpoonup_{B}\!(f\otimes g),\qquad g\pi=\varpi f.
\]
We denote by $\mathcal{YD}\mathrm{1C}(\mathsf{Vec}_\Bbbk)$ the category of Yetter--Drinfeld 1-cocycles thus obtained.
\end{definition}
\begin{remark}
    Observe that, choosing $a=b=1_{H}$ in \eqref{eq:1-cocycle}, we obtain $\pi(1_{H})=\pi(1_{H})\cdot_{A}\pi(1_{H})$. Moreover, taking $a=1_{H}$ in \eqref{comppiT}, we obtain $\pi(1_{H})\cdot_{A}\pi(1_{H})=1_{A}$. Hence
\begin{equation}\label{comp_pi_1}
    \pi(1_H) = 1_A
\end{equation} 
 is automatically satisfied. The same is true for 1-cocycles of Hopf algebras \cite{angiono2017hopf}.
\end{remark}
\begin{theorem}
    Let $H^{\bullet}:=(H,\bullet,1,\Delta,\epsilon,T)$ be a Hopf algebra. The following data are equivalent:
    \begin{enumerate}
        \item[\textit{i.}] a coalgebra $A$ and an isomorphism of coalgebras $\pi:H\to A$, such that $A$ with the coaction $\rho:=(\id\otimes\pi)\mathrm{Ad}_{L}\pi^{-1}$ is in $\mathrm{Bimon}(^{H^{\bullet}}_{H^{\bullet}}\mathcal{YD})$, and $\pi$ is a Yetter--Drinfeld 1-cocycle;
        \item[\textit{ii.}] a second operation $\cdot$ and a map $S$, providing a Yetter--Drinfeld brace structure on $H$. \end{enumerate}
This defines an equivalence of categories between the subcategory of $\mathcal{YD}\mathrm{1C}(\mathsf{Vec}_\Bbbk)$ given by Yetter--Drinfeld 1-cocycles $\pi:H\to A$ where $A$ has coaction $(\id\otimes\pi)\mathrm{Ad}_{L}\pi^{-1}$ and $\mathcal{YD}\mathrm{Br}(\mathsf{Vec}_\Bbbk)$.
\end{theorem}
\begin{proof} The proof is fundamentally the same as for Hopf braces \cite[Theorem 1.12]{angiono2017hopf}.\smallskip

\noindent\textbf{\textit{i.\@} to \textit{ii.}} Suppose given a Yetter--Drinfeld 1-cocycle $\pi\colon H\to A$. Define the second structure $H^{\raisebox{1pt}{.}}$ on $H$ by pulling back the multiplication $\cdot_{A}$ of $A$: \[a\cdot b:= \pi^{-1}(\pi(a)\cdot_{A}\pi(b)).\] This is clearly associative and, from \eqref{comp_pi_1}, it is immediate that $H^{\raisebox{1pt}{.}}$ is unitary with unit $1_H$. Define $S(a):= \pi^{-1}(a_1\rightharpoonup_A \pi T(a_2))$, and consequently let
\begin{align*}
    a\rightharpoonup b :\!&= S(a_1)\cdot (a_2\bullet b)\\
    &= \pi^{-1}\left( \pi S(a_1)\cdot_A \pi(a_2\bullet b) \right)\\
    \overset{\eqref{eq:1-cocycle}}&{=} \pi^{-1}\left( (a_1\rightharpoonup_A \pi T(a_2))\cdot_A \pi(a_3)\cdot_A (a_4\rightharpoonup_A \pi(b))  \right)\\
    \overset{\eqref{comppiT}}&{=}\pi^{-1}(a\rightharpoonup_A\pi(b)).
\end{align*}In particular $S$ satisfies $S(a) = a_1\rightharpoonup T(a_2)$, which is \eqref{def-S}. Since $\rightharpoonup$ is just the action $\rightharpoonup_A$ pulled back through $\pi$, this is clearly a left action of $H^\bullet$ on itself. By assumption, $A$ is in $\mathrm{Bimon}(^{H^{\bullet}}_{H^{\bullet}}\mathcal{YD})$ with coaction $\rho(a)=(\id\ot\pi)\mathrm{Ad}_{L}\pi^{-1}(a)=\pi^{-1}(a_{1})\bullet T(\pi^{-1}(a_{3}))\ot a_{2}$, hence $\rho$ and $\rightharpoonup$ satisfy the compatibility condition of a Yetter--Drinfeld module:
\[
\rho(a\rightharpoonup_{A}b)=a_{1}\bullet \pi^{-1}(b_{1})\bullet T(\pi^{-1}(b_{3}))\bullet T(a_{3})\ot(a_{2}\rightharpoonup_{A}b_{2}),
\]
moreover $\cdot_A$ and $\Delta$ satisfy the bialgebra compatibility in $^{H^{\bullet}}_{H^{\bullet}}\mathcal{YD}$:
\begin{equation}\label{compDeltadotA}
\Delta(a\cdot_{A}b)=a_{1}\cdot_{A}(\pi^{-1}(a_{2})\bullet T(\pi^{-1}(a_{4}))\rightharpoonup_{A} b_{1})\ot(a_{3}\cdot_{A} b_{2}).
\end{equation}
It easily follows that $(H^{\raisebox{1pt}{.}},\rightharpoonup,\mathrm{Ad}_{L})$ is in $^{H^{\bullet}}_{H^{\bullet}}\mathcal{YD}$, and that $m_{\cdot}$, $u$, $\Delta$ and $\epsilon$ are left linear with respect to $\rightharpoonup$. 

We show that $S$ is a convolution inverse of $\id_H$ with respect to the algebra structure $H^{\raisebox{1pt}{.}}$:
\begin{align*}
    a_{1}\cdot S(a_{2})&=a_{1}\cdot\pi^{-1}(a_{2}\rightharpoonup_{A} \pi T(a_{3}))\\&=\pi^{-1}(\pi(a_{1})\cdot_{A}(a_{2}\rightharpoonup_{A}\pi T(a_{3})))\\&=\pi^{-1}\pi(a_{1}\bullet T(a_{2}))\\&=\epsilon(a)1_{H},
\end{align*}
and 
\begin{align*}
    S(a_{1})\cdot a_{2}&=\pi^{-1}(a_{1}\rightharpoonup_{A}\pi T(a_{2}))\cdot a_{3}\\&=\pi^{-1}((a_{1}\rightharpoonup_{A}\pi T(a_{2}))\cdot_{A}\pi(a_{3}))\\\overset{\eqref{comppiT}}&{=}\pi^{-1}(\epsilon(a)1_{A})\\&=\epsilon(a)1_{H}.
\end{align*}
Let $a\leftharpoonup b:=T(a_{1}\rightharpoonup b_{1})\bullet a_{2}\bullet b_{2}=T\pi^{-1}(a_{1}\rightharpoonup_{A}\pi(b_{1}))\bullet a_{2}\bullet b_{2}$. Notice that \eqref{T-on-harpoon} is automatically satisfied. We verify that $\leftharpoonup$ satisfies \eqref{delta-colinear}:
\begin{align*}
(a_1\leftharpoonup b_1)\ot (a_2\leftharpoonup b_2) &=T\pi^{-1}(a_{1}\rightharpoonup_{A}\pi(b_{1}))\bullet a_{2}\bullet b_{2}\ot T\pi^{-1}(a_{3}\rightharpoonup_{A}\pi(b_{3}))\bullet a_{4}\bullet b_{4}\\
\overset{\eqref{compleftharpoondelta}}&{=}T\pi^{-1}(a_2\rightharpoonup_A \pi(b_2 ))\bullet a_3\bullet b_3\ot T\pi^{-1}(a_1 \rightharpoonup_A \pi(b_1))\bullet a_4\bullet b_4\\
&= (a\leftharpoonup b)_1\ot (a\leftharpoonup b)_2.
\end{align*}
Moreover, since $\epsilon(a\rightharpoonup b)=\epsilon(a)\epsilon(b)$, one clearly has $\epsilon(a\leftharpoonup b)=\epsilon(a)\epsilon(b)$, hence $\leftharpoonup$ is a morphism of coalgebras: by Remark \ref{rem:leftharpcoalg}, this is equivalent to requiring \eqref{matchedpair5}. Furthermore, $(a_1\rightharpoonup b_1)\bullet (a_2\leftharpoonup b_2) = a\bullet b$ holds by definition. By Lemma \ref{lemma:equiv-cond-colinearity}, we have that \eqref{compharpoonpbullet} is satisfied. 
Moreover, one has
\begin{align*}
    \pi\left( (a_1\bullet b)\cdot S(a_2)\cdot (a_3\bullet c) \right)& = \pi (a_1\bullet b)\cdot_{A} \pi(S(a_2))\cdot_{A} \pi(a_3\bullet c)\\
    \overset{\eqref{eq:1-cocycle}}&{=} \pi(a_1)\cdot_{A} (a_2\rightharpoonup \pi(b))\cdot_{A} \pi(S(a_3))\cdot_{A} \pi(a_4)\cdot_{A} (a_5\rightharpoonup \pi(c))\\
    &= \pi(a_1)\cdot_{A} (a_2\rightharpoonup \pi(b))\cdot_{A} (a_3\rightharpoonup \pi(c))\\
    &= \pi(a_1)\cdot_{A} \left(a_2\rightharpoonup (\pi(b)\cdot_{A} \pi(c))\right)\\
    &= \pi(a_1)\cdot_{A} (a_2\rightharpoonup \pi(b\cdot c))\\
    \overset{\eqref{eq:1-cocycle}}&{=}\pi(a\bullet(b\cdot c)),
\end{align*}
and this implies \eqref{compHopfbraces} because $\pi$ is bijective.

In order to obtain that $(H^{\raisebox{1pt}{.}},\rightharpoonup,\mathrm{Ad}_{L})$ is in $\mathrm{Bimon}(^{H^{\bullet}}_{H^{\bullet}}\mathcal{YD})$, it remains to prove that $\Delta$ is a morphism of algebras with respect to $\cdot$ in $^{H^{\bullet}}_{H^{\bullet}}\mathcal{YD}$, and that $m_{\cdot}$ is left colinear with respect to $\mathrm{Ad}_{L}$. Both easily follow from the analogous properties of $A$.

Finally, by Proposition \ref{prop:S-is-antipode}, we obtain that $S$ is an antipode for $H^{\raisebox{1pt}{.}}$. Thus, $(H,\cdot,\bullet,1,\Delta,\epsilon,S,T)$ is a Yetter--Drinfeld brace.
\smallskip

\noindent\textbf{\textit{ii.\@} to \textit{i.}} Suppose given a Yetter--Drinfeld brace $(H,\cdot,\bullet,1,\Delta,\epsilon, S,T)$. Define $A:=(H,\cdot, 1,\Delta,\epsilon)$. This is in $\mathrm{Bimon}(^{H^{\bullet}}_{H^{\bullet}}\mathcal{YD})$ with action $a\rightharpoonup b=S(a_{1})\cdot(a_{2}\bullet b)$, and the coaction $\mathrm{Ad}_{L}$. Set $\pi:=\mathrm{Id}_{H}\colon H^{\bullet}\to H^{\raisebox{1pt}{.}}$, which is clearly a coalgebra isomorphism. Moreover, we compute
\[
\pi(a\bullet b)=a\bullet b=a_{1}\cdot S(a_{2})\cdot(a_{3}\bullet b)=a_{1}\cdot (a_{2}\rightharpoonup b)=\pi(a_{1})\cdot(a_{2}\rightharpoonup\pi(b)),
\]
while \eqref{comppiT} and \eqref{compleftharpoondelta} are automatically satisfied. Thus $\pi:H^{\bullet}\to H^{\raisebox{1pt}{.}}$ is a Yetter--Drinfeld 1-cocycle. 
\smallskip

\noindent\textbf{Functoriality.} We verify that the above correspondence is functorial. Let $\pi \colon H\to A$, $\varpi\colon K\to B$ be Yetter--Drinfeld 1-cocycles, with a morphism $\pi\to \varpi $ in $\mathcal{YD}\mathrm{1C}(\mathsf{Vec}_\Bbbk)$ given by maps $f\colon H\to K$, $g\colon A\to B$. Then $f\colon H^\bullet\to K^\bullet $ is a morphism of Hopf algebras. Moreover,
\begin{align*}
    f(a\cdot_{H} b) &= f \pi^{-1}(\pi(a)\cdot_{A}\pi(b))\\
    &= \varpi^{-1} g(\pi(a)\cdot_{A}\pi(b))\\
    &= \varpi^{-1}(g\pi(a)\cdot_{B} g\pi(b))\\
    &= \varpi^{-1}(\varpi f(a)\cdot_{B} \varpi f(b))\\
    &= f(a)\cdot_{K} f(b)
\end{align*}
and
\begin{align*}
    S_{K} f(a) &= \varpi^{-1}(f(a)_1\rightharpoonup_{B} \varpi T_{K}(f(a)_2))\\
    &= \varpi^{-1}(f(a_1)\rightharpoonup_{B} \varpi T_{K}(f(a_2)))\\
    &=\varpi^{-1}(f(a_1)\rightharpoonup_{B} \varpi f T_{H}(a_2))\\
    &= \varpi^{-1}(f(a_1)\rightharpoonup_{B} g\pi T_{H}(a_2))\\
    &= \varpi^{-1} g(a_1\rightharpoonup_A \pi T_H(a_2))\\
    &= f\pi^{-1}(a_1\rightharpoonup_A \pi T_H(a_2))\\
    &=  f (a_1\rightharpoonup_H T_H(a_2))\\
    &= f S_H(a),
\end{align*}
thus $f$ is a morphism in $\mathcal{YD}\mathrm{Br}(\mathsf{Vec}_\Bbbk)$. 
The above construction clearly sends identity morphisms into identity morphisms, and respects the composition of maps. We denote by $F$ the functor $\mathcal{YD}\mathrm{1C}(\mathsf{Vec}_\Bbbk)\to \mathcal{YD}\mathrm{Br}(\mathsf{Vec}_\Bbbk)$ thus obtained.

Conversely, a functor $G\colon \mathcal{YD}\mathrm{Br}(\mathsf{Vec}_\Bbbk)\to \mathcal{YD}\mathrm{1C}(\mathsf{Vec}_\Bbbk)$ is defined by the following action on the maps: if $f\colon H\to K$ is a morphism of Yetter--Drinfeld braces, let $Gf$ be the pair of maps $(f,f)$, where the former is interpreted as a map $f\colon H^\bullet \to K^\bullet$, and the latter as a map $f\colon H^{\raisebox{1pt}{.}}\to K^{\raisebox{1pt}{.}}$. We check that this is a morphism in $\mathcal{YD}\mathrm{1C}(\mathsf{Vec}_\Bbbk)$ between $\id_H\colon H^\bullet \to H^{\raisebox{1pt}{.}}$ and $\id_K\colon K^\bullet\to  K^{\raisebox{1pt}{.}}$. 

We already know that $f:H^{\bullet}\to K^{\bullet}$ is a morphism of Hopf algebras and $f:H^{\raisebox{1pt}{.}}\to K^{\cdot}$ is a morphism of algebras and coalgebras. Moreover, $f(a\rightharpoonup_H b) = f(a)\rightharpoonup_H f(b)$ holds true because $f$ is a morphism in $\mathcal{YD}\mathrm{Br}(\mathsf{Vec}_\Bbbk)$. The assignment $G$ clearly respects identities and compositions.
\smallskip

\noindent\textbf{Equivalence.} We finally check that this is an equivalence of categories. It is clear that $FG$ is the identity functor. Given a Yetter--Drinfeld 1-cocycle $\pi\colon H^{\bullet}\to A$, one has that $GF(\pi)$ is the Yetter--Drinfeld 1-cocycle $\id\colon H^\bullet \to H^{\raisebox{1pt}{.}}$, which is isomorphic to $\pi\colon H^{\bullet}\to A$ via the morphism $(\id_H,\pi^{-1})$. It is easy to verify that this is indeed a morphism in $\mathcal{YD}\mathrm{1C}(\mathsf{Vec}_\Bbbk)$, and its inverse is clearly $(\id_H, \pi)$.
\end{proof}

\section{Coquasitriangular Hopf algebras and Yetter--Drinfeld braces}\label{section:coquasi}\noindent
In this section, we shall prove that coquasitriangular bialgebras $H$ yield matched pairs of actions $(H,\rightharpoonup,\leftharpoonup)$ satisfying the hypotheses of Theorem \ref{thm:matched-pair-yields-YDbrace}, and hence enable us to construct concrete examples of Yetter--Drinfeld braces. The additional structure, here, coincides with Majid's transmutation of $H$ \cite{Maj2}.
\subsection{Coquasitriangular bialgebras}\label{subsection:coquasi} The notion of \emph{quasitriangular bialgebra} was introduced by Drinfeld \cite{drinfeld1989quasi}, to relax the hypothesis of cocommutativity. Here we shall operate with the dual notion, which appeared in several works in between 1990 and 1993, such as in Schauenburg \cite{Sch}, Doi \cite{doi:10.1080/00927879308824649}, Larson and Towber \cite{larson1991two}, and Majid \cite{MJbraided}.
\begin{definition}
    A \emph{coquasitriangular bialgebra}, denoted by $(H,m,u,\Delta,\epsilon, \Rr)$ or simply by $(H,\Rr)$, is the datum of a bialgebra $(H,m,u,\Delta,\epsilon)$ and of a convolution-invertible morphism $\Rr\colon H\ot H\to \Bbbk$, satisfying:
    \begin{align}
     \label{coQT.1}\tag{\sc coqt.1}&\Rr(a_1\otimes b_{1})a_2 b_2 = b_1 a_1 \Rr(a_2\otimes b_{2});\\
    \label{coQT.2}\tag{\sc coqt.2}&\Rr(a\otimes bc)=\Rr(a_1\otimes c)\Rr(a_2\otimes b);\\
    \label{coQT.3}\tag{\sc coqt.3}&\Rr(ab\otimes c)=\Rr(a\otimes c_1)\Rr(b\otimes c_2).
\end{align}
Moreover, $(H,\Rr)$ is called \emph{cotriangular} if $\Rr^{-1}=\Rr^{\mathrm{op}}$, where $\Rr^{-1}$ denotes the convolution inverse of $\Rr$.
\end{definition}

If $\Rr$ provides a coquasitriangular structure, then the convolution inverse $\Rr^{-1}$ satisfies:
\begin{align}
\label{coQT.1'}\tag{\sc coqt.1$'$}&a_1b_1\Rr^{-1}(a_2\ot b_2) = \Rr^{-1}(a_1\ot b_1)b_2 a_2;\\
    \label{coQT.2'}\tag{\sc coqt.2$'$}&\Rr^{-1}(a\otimes bc)=\Rr^{-1}(a_{1}\otimes b)\Rr^{-1}(a_{2}\otimes c);\\
    \label{coQT.3'}\tag{\sc coqt.3$'$}&\Rr^{-1}(ab\otimes c)=\Rr^{-1}(b\otimes c_{1})\Rr^{-1}(a\otimes c_{2}).
\end{align}

Let $\Rr$ be a coquasitriangular structure on a bialgebra $H$. Then, it is known that $\Rr$ and its convolution inverse $\Rr^{-1}$ satisfy the \emph{quantum Yang--Baxter equation} (see \cite[Lemma 2.2.3]{Majid-book}):
    \begin{equation}
        \label{YBE}\tag{\sc ybe}
            \Rr^{\pm 1}(a_1\ot b_1)\Rr^{\pm 1}(a_2\ot c_1)\Rr^{\pm 1}(b_2\ot c_2) = \Rr^{\pm 1}(b_1\ot c_1)\Rr^{\pm 1}(a_1\ot c_2)\Rr^{\pm 1}(a_2\ot b_2).
    \end{equation}
Moreover, if $H$ is a Hopf algebra with antipode $S$, then the following relations hold true:
\[
\Rr(S(a)\ot b)=\Rr^{-1}(a\ot b),\quad \Rr^{-1}(a\ot S(b))=\Rr(a\ot b),\quad \Rr^{\pm1}(S(a)\ot S(b))=\Rr(a\ot b),
\]
see \cite[Lemma 2.2.2]{Majid-book}. It is well known that a bialgebra $H$ is coquasitriangular if and only if the category $\vect^{H}$ of right $H$-comodules is braided (see \cite[Exercise 9.2.9]{Majid-book} and the considerations following \cite[Theorem 9.2.4]{Majid-book}). Indeed, given a coquasitriangular structure $\Rr$, the bialgebra $H$ becomes a braided object with the braiding operator
\[\sigma^{\Rr}_{H,H}(a\ot b):= b_1\ot a_1 \Rr(a_2\ot b_2),
\]and a braiding on $\vect^{H}$ is defined for all $X,Y$ in $\vect^{H}$ by setting 
\begin{equation}\label{braiding-comodules}
\sigma^{\Rr}_{X,Y}:X\otimes Y\to Y\otimes X,\quad x\otimes y\mapsto y_{0}\otimes x_{0}\Rr(x_{1}\otimes y_{1}).    
\end{equation}
Moreover, $H$ is cotriangular if and only if the category $\vect^{H}$ is symmetric, i.e.\@ $(\sigma^{\Rr}_{X,Y})^{-1}=\sigma^{\Rr}_{Y,X}$ for all $X,Y$ in $\vect^{H}$.

In the rest of this section, however, we shall use another braiding operator on $H$. Let $(H,\Rr)$ be a coquasitriangular bialgebra, then the map $\sigma_{H,H}\colon H\ot H\to H\ot H$ defined by
\begin{equation}\label{braiding-bicomodules}
    \sigma_{H,H}:H\ot H\to H\ot H,\ a\ot b\mapsto\Rr^{-1}(a_{1}\ot b_{1})b_{2}\ot a_{2}\Rr(a_{3}\ot b_{3}),
\end{equation}
is known to be a braiding operator on $H$, with inverse given by
\[
\sigma^{-1}_{H,H}:H\ot H\to H\ot H,\ a\ot b\mapsto\Rr(b_{1}\ot a_{1})b_{2}\ot a_{2}\Rr^{-1}(b_{3}\ot a_{3}).
\]
Moreover, $m_{H}\sigma_{H,H}= m_{H}$ holds.
\begin{remark}
    The braiding operator \eqref{braiding-bicomodules} extends to a braiding
\[ x\ot y\mapsto \Rr^{-1}(x_{-1}\ot y_{-1}) y_0 \ot x_0 \Rr(x_1\ot y_1)\] on the category of \emph{bicomodules} $^H\!\vect^H$. However, the braidings on $(^H\!\vect^H,\otimes,\Bbbk)$ are not yet classified, and the problem is seemingly hard. When the category of \emph{bimodules} ${}_H\!\vect_H$ is endowed with the tensor product $\ot_H$ instead of $\ot_\Bbbk$, however, a classification was found by Agore, Caenepeel and Militaru \cite{agore2014braidings}, which could be dualised for the category $^{H}\!\vect^{H}$.
\end{remark}
\begin{remark}\label{remark:involutive}
Let $\sigma_{H,H}$ be the braiding operator defined in \eqref{braiding-bicomodules}. Notice that 
\[
\sigma_{H,H}^{2}=\Rr^{-1}*(\Rr^{-1})^{\mathrm{op}}*\mathrm{Id}_{H\ot H}*\Rr^{\mathrm{op}}*\Rr=(\Rr^{\mathrm{op}}*\Rr)^{-1}*\mathrm{Id}_{H\ot H}*(\Rr^{\mathrm{op}}*\Rr),
\]
hence we have that $\sigma_{H,H}^{2}=\mathrm{Id}_{H\ot H}$ if and only if $\mathrm{Id}_{H\ot H}*(\Rr^{\mathrm{op}}*\Rr)=(\Rr^{\mathrm{op}}*\Rr)*\mathrm{Id}_{H\ot H}$. The element $\Rr^\op *\Rr$ is also known as the \emph{quantum Killing form}. In particular, if $\Rr^{\mathrm{op}}=\Rr^{-1}$, one clearly has $\sigma^{-1}_{H,H}=\sigma_{H,H}$. 
\end{remark}
\subsection{Coquasitriangular Hopf algebras and matched pairs of actions}\label{subsection:coQT-and-MP} Suppose given a coquasitriangular Hopf algebra $(H,\bullet, 1,\Delta,\epsilon, T,\Rr)$. We shall see that the braiding operator \eqref{braiding-bicomodules} yields a matched pair satisfying \eqref{braided-commutativity}, and hence a Yetter--Drinfeld brace. The product and the antipode of the corresponding Hopf algebra structure in $^H_H\mathcal{YD}$ coincide with those obtained by \emph{transmutation} of $H$, introduced by Majid \cite{Maj2}; thus providing a Yetter--Drinfeld brace-theoretic interpretation of the theory of transmutation. At the end of this section, we shall also observe that the same construction cannot be carried out with the braiding \eqref{braiding-comodules}. 

Let $\sigma_{H,H}$ be defined as in \eqref{braiding-bicomodules}, and define 
\begin{align}
\label{R-harpoon-left}&a\rightharpoonup b=(\id\ot \epsilon)\sigma_{H,H}(a\ot b) = \Rr^{-1}(a_{1}\ot b_{1})b_{2}\Rr(a_{2}\ot b_{3});\\
\label{R-harpoon-right}&a\leftharpoonup b=(\epsilon\ot \id)\sigma_{H,H}(a\ot b) = \Rr^{-1}(a_{1}\ot b_{1})a_{2}\Rr(a_{3}\ot b_{2}).
\end{align}
These are a left and a right action respectively, by Lemma \ref{lemma:retrieve-actions}, and $H$ becomes a left and right module coalgebra with the respective actions. Moreover, these actions provide a matched pair $(H,H)$, see \cite[Example 7.2.7]{Majid-book}. Observe that the braiding \eqref{braiding-bicomodules} is retrieved as 
\[\sigma_{H,H}(a\ot b) = (a_1\rightharpoonup b_1)\ot (a_2\leftharpoonup b_2).\]
Hence the fact that $(H,\rightharpoonup,\leftharpoonup)$ satisfies \eqref{matchedpair-1}--\eqref{matchedpair-4} can also be obtained using Lemma \ref{lem:mhex} and Lemma \ref{lemma:matchedpair-iff-braiding}, while \eqref{braided-commutativity} is immediate.
\begin{remark}
The bialgebra $H \bowtie H$ described in \S$\,$\ref{section:prelim}, when the matched pair is derived as in \eqref{R-harpoon-left} and \eqref{R-harpoon-right}, is denoted in Majid \cite{Majid-book} by $H\bowtie_{\Rr}H$. This has product given by \[(a\otimes h)(b\otimes g)=\Rr^{-1}(h_{1}\otimes b_{1})a\bullet b_{2}\otimes h_{2}\bullet g\Rr(h_{3}\otimes b_{3}).\] Moreover, we have
\[
a\leftharpoonup b=\Rr^{-1}(a_{1}\otimes b_{1})a_{2}\Rr(a_{3}\otimes b_{2})=\Rr(T(a_{1})\otimes b_{1})a_{2}\Rr(a_{3}\otimes b_{2})=a_{2}\Rr(T(a_{1})\bullet a_{3}\otimes b),
\]
so this situation falls under the hypotheses of \cite[Lemma 7.4.8]{Majid-book} where, in our case, the coaction is given by $\mathrm{Ad}_{R}\colon a\mapsto a_{2}\otimes T(a_{1})\bullet a_{3}$. Thus \cite[Lemma 7.4.8]{Majid-book} ensures us that $H$ is also a right $H$-module with respect to $\leftharpoonup$. Hence, using $\leftharpoonup$ and $\mathrm{Ad}_{R}$ one can construct the bialgebra $ H \rbiprod\, \underline{H}  $ (we borrow the notation from \cite{Majid-book}), where $\underline{H}$ is the (right) transmutation of $H$ \cite[Theorem 7.4.1]{Majid-book} (see also \cite[Example 9.4.10]{Majid-book}), and the symbol $\rbiprod$ denotes the bialgebra with smash product algebra structure, and smash coproduct coalgebra structure. So $H \rbiprod\, \underline{H}$ is exactly the bosonisation $H\#\underline{H}$, built using the right action and the right coaction. By \cite[Theorem 7.4.10]{Majid-book}, one also obtains an isomorphism $H\bowtie_{\Rr}H\cong H \rbiprod\, \underline{H} $. In Remark \ref{rem:transmutation}, we shall relate $H^{\raisebox{1pt}{.}}$ with a \emph{left} version $\overline{H}$ of Majid's transmutation.
\end{remark}

Notice that 
\[
a\rightharpoonup b=\Rr^{-1}(a_{1}\otimes b_{1})b_{2}\Rr(a_{2}\otimes b_{3})=\Rr^{-1}(a_{1}\otimes b_{1})b_{2}\Rr^{-1}(a_{2}\otimes T(b_{3}))=b_{2}\Rr^{-1}(a\otimes b_{1}\bullet T(b_{3})).
\]

In this case, \eqref{def-cdot} takes the following form:
\begin{align}
    \label{cdot_R_alt}a\cdot b = a_1\bullet b_2\Rr^{-1}(T(a_2)\ot (b_1\bullet T(b_3))),
\end{align}
and also
\begin{equation}\label{cdot_R}\begin{split}a\cdot b &= \Rr^{-1}(T(a_{3})\ot b_{1})a_{1}\bullet b_{2}\Rr^{-1}(T(a_{2})\ot T(b_{3}))\\ &=\Rr^{-1}(T(a_{3})\ot b_{1})a_{1}\bullet b_{2}\Rr^{-1}(a_{2}\ot b_{3})\\
 \overset{\eqref{coQT.1'}}&{=} \Rr^{-1}(T(a_3)\ot b_1)\Rr^{-1}(a_1\ot b_2)b_3\bullet a_2\\
&=\Rr^{-1}((a_1\bullet T(a_3))\ot b_1)b_2\bullet a_2,
\end{split}\end{equation}
while \eqref{def-S} has the following form:\begin{equation}\label{S_R}\begin{split}    S(a)&=a_{1}\rightharpoonup T(a_{2})=T(a_{3})\Rr^{-1}(a_{1}\ot T(T(a_{2})\bullet a_{4}))\\&=T(a_{3})\Rr(a_{1}\ot T(a_{2})\bullet a_{4})\\&=T(a_{4})\Rr(a_{1}\ot a_{5})\Rr(a_{2}\ot T(a_{3})).
\end{split}\end{equation}
\begin{remark}
    When $H$ is cocommutative, $\rightharpoonup$ is trivial, and hence $\cdot$ coincides with $\bullet$. Thus the corresponding Yetter--Drinfeld brace is just the \emph{trivial Hopf brace} $(H,\bullet,\bullet)$ (the terminology is again adapted from \cite{guarnieri2017skew}).
\end{remark}
\begin{theorem}\label{thm:RYDB}
    Let $(H,\bullet,1,\Delta,\epsilon,T, \Rr)$ be a coquasitriangular Hopf algebra. Define $\rightharpoonup$ and $\leftharpoonup$ as above, and define $\cdot$ and $S$ accordingly. Then, $(H,\cdot,\bullet,1,\Delta,\epsilon, S,T)$ is a Yetter--Drinfeld brace. Furthermore, if $\Rr$ is cotriangular, then $\cdot$ is braided-commutative with respect to the braiding of $^H_H\mathcal{YD}$, whence also $S^{2}=\id_{H}$. 
\end{theorem}
\begin{proof}
    We already know that $(H,\rightharpoonup,\leftharpoonup)$ satisfies \eqref{matchedpair-1}--\eqref{matchedpair-4} \cite[Example 7.2.7]{Majid-book}. Moreover, it follows from \eqref{coQT.1} that $H$ is braided-commutative, i.e.\@ $m_\bullet \sigma_{H,H} = m_\bullet$, which translates into condition \eqref{braided-commutativity}. Corollary \ref{cor:Hyetterdrinfeldbrace} yields that $(H,\cdot,\bullet,1,\Delta,\epsilon,S,T)$ is a Yetter--Drinfeld brace. 
    
    Suppose now that $\Rr$ is cotriangular. The braided commutativity of $\cdot$ with respect to the braiding of $^{H}_{H}\mathcal{YD}$ is expressed by the condition $a\cdot b=(a_{1}\bullet T(a_{3})\rightharpoonup b)\cdot a_{2}$, which in view of \eqref{cdot_R} is rewritten as    
\[ b_3\bullet a_2 \Rr^{-1}(T(a_3)\ot b_1)\Rr^{-1}(a_1\ot b_2) = b_2\cdot  a_3\Rr^{-1}((a_1\bullet T(a_5))\ot b_1)\Rr((a_2\bullet T(a_4))\ot b_3).\]    
    Suppose $\Rr^{-1} = \Rr^\op$. Then
    \begin{align*}
        &\hspace{-2em}b_2 \cdot a_3\Rr^{-1}((a_1\bullet T(a_5))\ot b_1)\Rr((a_2\bullet T(a_4))\ot b_3)\\
        &= b_3\bullet a_4 \Rr^{-1}(T(a_7)\ot b_1)\Rr^{-1}(a_1\ot b_2)\Rr^{-1}(T(b_5)\ot a_3)\\
        &\hspace{1em}\;\Rr^{-1}(b_4\ot a_5)\Rr(a_2\ot b_6)\Rr(T(a_6)\ot b_7)\\
        &= b_3\bullet a_4 \Rr^{-1}(T(a_7)\ot b_1) \Rr^{-1}(a_1\ot b_2)\Rr(a_3\ot T(b_5))\\
        &\hspace{1em}\;\Rr(a_5\ot b_4)\Rr(a_2\ot b_6)\Rr(T(a_6)\ot b_7)\\
        &= b_3\bullet a_2 \Rr^{-1}(T(a_5)\ot b_1)\Rr^{-1}(a_1\ot b_2) \Rr(a_3\ot b_4)\Rr(T(a_4)\ot b_5)\\
        &= b_3\bullet a_2 \Rr^{-1}(T(a_3)\ot b_1)\Rr^{-1}(a_1\ot b_2),
    \end{align*}
    as desired. Given a Hopf algebra $(H,\cdot, 1, \Delta, \epsilon, S)$ in a braided monoidal category $\Mm$, it is known that, if $\cdot$ is braided-commutative in $\Mm$, then $S$ is involutive.
\end{proof}
The Hopf algebra SL\textsubscript{\emph{q}}(2) with $q=-1$ will be an example of a coquasitriangular Hopf algebra with involutive antipode $S$, which is not cotriangular.

\begin{remark}\label{rem:transmutation} Let us show how the Hopf algebra $(H,\cdot,1,\Delta,\epsilon,S)$ is related to the transmutation of the Hopf algebra $(H,\bullet,1,\Delta,\epsilon,T)$. For more details about transmutation theory we redirect the reader to Majid \cite{Maj2, Majid-book}. 
Let $(H,\bullet,1,\Delta,\epsilon,T,\Rr)$ be a coquasitriangular Hopf algebra, so that $\Rr^{-1}$ induces a braiding on the category $^{H}\mathfrak{M}$. It is possible to define a second product operation $\overline{\bullet}$ on $H$, and a second antipode $\overline{T}$, such that they provide, on $H$, a Hopf algebra structure $\overline{H}$ in $^{H}\mathfrak{M}$ (with same comultiplication $\Delta$ and counit $\epsilon$), where the comodule structure is given by $\mathrm{Ad}_L$. This is called the \emph{transmutation} of $H$; see \cite[dual of Example 9.4.10]{Majid-book}. The transmuted product and the transmuted antipode are, respectively,
\begin{align*}&a\,\overline{\bullet}\, b:= \Rr^{-1}(T(a_2)\ot(b_1\bullet T(b_3))) a_1\bullet b_2,\\
&\overline{T}(a):= \Rr^{-1}(a_1\ot T(T(a_2)\bullet a_4)) T(a_3).\end{align*}
By \eqref{cdot_R_alt} we have $a\,\overline{\bullet}\, b=a\cdot b$ and by \eqref{S_R} we have $\overline{T}(a)=S(a)$, so $(H,\cdot,1,\Delta,\epsilon,S)$ coincides with the transmutation of $(H,\bullet,1,\Delta,\epsilon,T)$. Hence the transmutation of coquasitriangular Hopf algebras gives us examples of Yetter--Drinfeld braces. Theorem \ref{thm:matched-pair-yields-YDbrace} implies, in particular, that $\overline{H}$ lies in $\mathrm{Hopf}({}^{H^{\bullet}}_{H^{\bullet}}\mathcal{YD})$, with action given by $\rightharpoonup$. Observe from \cite[duals of Theorem 4.8 and Corollary 4.9]{ZhuRelativeYD} (with $f=\id_{H^{\bullet}}$) that there is a braided monoidal equivalence between $_{H^{\raisebox{1pt}{.}}}(^{H^{\bullet}}\mm)$ and ${}^{H^{\bullet}}_{H^{\bullet}}\mathcal{YD}$, hence one can also recover an object in $\mathrm{Hopf}(_{H^{\raisebox{1pt}{.}}}(^{H^{\bullet}}\mm))$. 

\end{remark}

\begin{remark}
    We know from Counterexample \ref{counterexample:nomatchedpair} that the following two actions do not form a matched pair:
\[
a\rightharpoonup b=\Rr(a\otimes b_{2})b_{1},\qquad a\leftharpoonup b=a_{1}\Rr(a_{2}\otimes b).
\]
We may define $\cdot$ as in \eqref{def-cdot}:
\[
a\cdot b :=a_{1}\bullet(T(a_{2})\rightharpoonup b)=a_{1}\bullet \Rr(T(a_{2})\otimes b_{2})b_{1}=a_{1}\bullet b_{1}\Rr^{-1}(a_{2}\otimes b_{2}).
\]
With this product, $(H,\cdot, 1)$ is an algebra. But now observe that, since $\epsilon(a\cdot b)=\Rr^{-1}(a\ot b)$, $\epsilon$ is not a morphism of algebras with respect to $\cdot$, unless $\Rr^{-1}=\epsilon\otimes\epsilon$ and hence $\Rr=\epsilon\otimes\epsilon$: in this case, by \eqref{coQT.1}, $H$ is commutative.
\end{remark}
In the case of cocommutative Hopf braces, it is known (see \cite[Corollary 2.5]{angiono2017hopf}) that the additional operation \eqref{def-cdot} is commutative if and only if the braiding operator $\sigma(a\ot b):= (a_1\rightharpoonup b_1)\ot (a_2\leftharpoonup b_2)$ is involutive. An analogous result for Yetter--Drinfed braces may possibly hold, substituting commutativity with \emph{braided commutativity} in $^H_H\mathcal{YD}$. However, we attempted mimicking the proof of \cite[Theorem 2.3]{angiono2017hopf} in several straightforward ways, and all of them failed; thus leaving the following problem open:
\begin{problem}
    Let $(H,\rightharpoonup,\leftharpoonup)$ be a matched pair of actions on a Hopf algebra, and let $\cdot$ be defined as in \eqref{def-cdot}. Is it true that $\sigma$ is involutive if and only if $\cdot$ is braided-commutative?
\end{problem}

\begin{remark}
The well-known FRT construction \cite{FRT} allows one to construct a coquasitriangular bialgebra $A(c)$ from a solution $c\in\mathrm{End}_{\Bbbk}(V\otimes V)$ of the \emph{braid equation} $(c\ot \id)(\id\ot c)(c\ot \id)=(\id\ot c)(c\ot \id)(\id\ot c)$. The bialgebra $A(c)$ is spanned by indeterminates $t^i_j$ (see the construction as presented by Kassel \cite{kassel2012quantum}). 
If we adjoin to $A(c)$ inverses $\bar{t}$ to the variables, one can canonically extend the relations, the comultiplication and the counit to $\bar{t}$, and define $St := \bar{t}$, obtaining a Hopf algebra $GL(c)$. Suppose now that $\Rr$ extends so that $GL(c)$ is coquasitriangular (this happens in many concrete examples, see \cite[\S$\,$4.2]{Majid-book}). Using the coquasitriangular Hopf algebra $(GL(c),\Rr)$ we can obtain a Yetter--Drinfeld brace, where $\cdot$ is defined as in \eqref{cdot_R}, and $S$ as in and \eqref{S_R}.
\end{remark}

\section{Examples}\label{section:example}
\noindent In this section, we apply our results to known examples of coquasitriangular Hopf algebras; namely the Sweedler's Hopf algebra, the algebras $E(n)$, the Hopf algebra $\mathrm{SL}_q(2)$, and an example in the class of Suzuki algebras. For each example, we describe the corresponding Yetter--Drinfeld brace. Throughout this section, our notation changes from ``brace-theoretic'' $(H, \cdot,\bullet, 1,\Delta, \epsilon, S,T)$ to ``transmutation-theoretic'' $(H,\overline{\cdot}, \cdot, 1,\Delta, \epsilon,\overline{S},S)$, in order to match the way in which the objects in the examples are usually presented.
\subsection{The Sweedler's Hopf algebra}\label{subsection:Sweedler} Let us suppose $\mathrm{char}(\Bbbk)\not=2$. The Sweedler's Hopf algebra $H_4 = E(1)$ is the free $\Bbbk$-algebra generated by two elements $g$ and $x$ modulo the relations 
    \[
    g^{2}=1,\ x^{2}=0,\ \text{and}\ xg=-gx, 
    \]   
    where $1$ denotes the unit of $H_4$. It becomes a Hopf algebra with comultiplication, counit and antipode given by $\Delta(g)=g\otimes g$, $\Delta(x)=x\otimes1+g\otimes x$, $\epsilon(g)=1$, $\epsilon(x)=0$, $S(g)=g$ and $S(x)=xg$, respectively. In particular, $H_4$ is neither commutative nor cocommutative. We recall that $H_4$ has a one-parameter family of cotriangular structures given by
\begin{equation*}
\Rr\begin{pmatrix}
    1\otimes1 & 1\otimes g
    &1\otimes x & 1\otimes xg\\
    g\otimes1 & g\otimes g & g\otimes x & g\otimes xg\\
    x\otimes1 & x\otimes g & x\otimes x & x\otimes xg\\
    xg\otimes1 & xg\otimes g & xg\otimes x & xg\otimes xg
\end{pmatrix}=\begin{pmatrix}
    1 & 1 & 0 & 0\\
    1 & -1 & 0 & 0\\
    0 & 0 & k & -k\\
    0 & 0 & k & k
\end{pmatrix}, 
\end{equation*}
where $k$ is a parameter in $\Bbbk$ (see \cite[Example 2.2.6]{Majid-book}). The convolution inverse of $\Rr$ is $\Rr^{-1}=\Rr^{\mathrm{op}}$ given by the transpose matrix
\begin{equation*}
\Rr^{-1}\begin{pmatrix}
    1\otimes1 & 1\otimes g
    &1\otimes x & 1\otimes xg\\
    g\otimes1 & g\otimes g & g\otimes x & g\otimes xg\\
    x\otimes1 & x\otimes g & x\otimes x & x\otimes xg\\
    xg\otimes1 & xg\otimes g & xg\otimes x & xg\otimes xg
\end{pmatrix}=\begin{pmatrix}
    1 & 1 & 0 & 0\\
    1 & -1 & 0 & 0\\
    0 & 0 & k & k\\
    0 & 0 & -k & k
\end{pmatrix}.
\end{equation*}
Given $\alpha:=a_{1}1+a_{2}g+a_{3}x+a_{4}xg$ in $H_4$, one has
\begin{align*}
(\Delta\ot\id)\Delta(\alpha)&=a_{1}1\ot1\ot1+a_{2}g\ot g\ot g+a_{3}x\ot1\ot1+a_{3}g\ot x\ot1\\&\hspace{1em}\;+a_{3}g\ot g\ot x+a_{4}xg\ot g\ot g+a_{4}1\ot xg\ot g+a_{4}1\ot1\ot xg.
\end{align*}
Hence, given $\beta:=b_{1}1+b_{2}g+b_{3}x+b_{4}xg$ in $H_4$, one has 
\begin{align*}
\alpha\rightharpoonup\beta&=(a_{1}b_{1}+a_{2}b_{1}+k(a_{3}b_{3}-a_{3}b_{4}-a_{4}b_{3}+a_{4}b_{4}))1\\&\hspace{1em}\;+(a_{1}b_{2}+a_{2}b_{2}+k(a_{3}b_{4}-a_{3}b_{3}-a_{4}b_{4}+a_{4}b_{3}))g\\&\hspace{1em}\;+(a_{1}b_{3}-a_{2}b_{3})x+(a_{1}b_{4}-a_{2}b_{4})xg.
\end{align*}
Analogously, for $\leftharpoonup$ one has:
\begin{align*}\alpha\leftharpoonup \beta &= (a_1b_1+a_1b_2+a_3b_3 k+a_3b_4 k+a_4 b_3 k+a_4b_4 k)1 \\
&\hspace{1em}\;+(a_2b_1+a_2b_2-a_3b_3 k-a_3b_4 k-a_4b_3k -a_4b_4 k)g\\&\hspace{1em}\;
+(a_3b_1-a_3b_2)x+(a_4b_1-a_4b_2)xg.\end{align*}
The two actions are summarised in Table \ref{table:actions}.
\begin{table}[t]\centering
\begin{subtable}[t]{0.4\linewidth}
\begin{tabular}{c|cccc}
    $\rightharpoonup$& $1$ & $g$ & $x$ & $xg$\\ \hline
    $1$ & $1$ & $g$ & $x$ & $xg$ \\ 
    $g$ & $1$ & $g$ & $-x$ & $-xg$ \\ 
    $x$ & $0$ & $0$ & $k1-kg$ & $kg-k1$ \\ 
    $xg$ & $0$ & $0$ & $kg-k1$ & $k1-kg$ \\ 
    \end{tabular}\caption{Left action.}\end{subtable} 
\begin{subtable}[t]{0.4\linewidth}\begin{tabular}{c|cccc}
    $\leftharpoonup$& $1$ & $g$ & $x$ & $xg$\\ \hline
    $1$ & $1$ & $1$ & $0$ & $0$ \\ 
    $g$ & $g$ & $g$ & $0$ & $0$ \\ 
    $x$ & $x$ & $-x$ & $k1-kg$ & $k1-kg$ \\ 
    $xg$ & $xg$ & $-xg$ & $k1-kg$ & $k1-kg$ \\ 
    \end{tabular}\caption{Right action.}\end{subtable}\caption{Left and right actions for $H_4$.}\label{table:actions}
    \end{table}
We denote by $\,\overline{\cdot}\,$ the transmuted multiplicative structure on $H_4$, defined according to \eqref{cdot_R}. This is explicitly given by
\begin{align*}
    \alpha\,\overline{\cdot}\, \beta&= \alpha_1 (S(\alpha_2)\rightharpoonup\beta)\\
    &= a_1(S(1)\rightharpoonup \beta)+a_2g(S(g)\rightharpoonup \beta)+a_3 x(S(1)\rightharpoonup \beta)\\&\hspace{1em}\;+a_3 g(S(x)\rightharpoonup \beta)+a_4 xg(S(g)\rightharpoonup \beta)+a_4(S(xg)\rightharpoonup \beta)\\
    &= a_1\beta +a_2 g (b_11+b_2 g -b_3 x-b_4 xg)+a_3 x\beta +a_3 g((kb_4-kb_3)1+(kb_3-kb_4)g)\\&\hspace{1em}\;+a_4 xg(b_11+b_2 g -b_3 x-b_4 xg)+a_4((kb_4-kb_3)1+(kb_3-kb_4)g)\\
    &=(a_1b_1+a_2b_2+ka_3 b_3-ka_3b_4+ka_4b_4-ka_4 b_3)1\\&\hspace{1em}\; +(a_1b_2+a_2b_1+ka_3b_4-ka_3 b_3+ka_4b_3-ka_4b_4)g\\&\hspace{1em}\; +(a_1b_3+a_2b_4+a_3b_1+a_4b_2)x\\&\hspace{1em}\;+(a_1b_4+a_2b_3+a_3b_2+a_4b_1)xg.
\end{align*}
The multiplication $\,\overline{\cdot}\,$ is described in Table \ref{table:Sweedler}. The algebra $(H_4, \,\overline{\cdot}\, ,1)$ is generated by $g,x$ modulo the relations $g \,\overline{\cdot}\, g =1$, $x\,\overline{\cdot}\, x = k1-kg$, $x\,\overline{\cdot}\, g= g\,\overline{\cdot}\, x$.
\begin{table}[t]\centering
\begin{tabular}{c|cccc}
    $\,\overline{\cdot}\,$& $1$ & $g$ & $x$ & $xg$\\ \hline
    $1$ & $1$ & $g$ & $x$ & $xg$ \\ 
    $g$ & $g$ & $1$ & $xg$ & $x$ \\ 
    $x$ & $x$ & $xg$ & $k1-kg$ & $kg-k1$ \\ 
    $xg$ & $xg$ & $x$ & $kg-k1$ & $k1-kg$ \\ 
    \end{tabular}
    \caption{Multiplication table for $\,\overline{\cdot}\,$ in $H_4$.}\label{table:Sweedler}
\end{table}
The new antipode $\overline{S}$ is defined by $\overline{S}(\alpha):=\alpha_{1}\rightharpoonup S(\alpha_{2})$ as in \eqref{S_R}, whence $\overline{S}(1)=1$, $\overline{S}(g)=g$, $\overline{S}(x)= -xg=-x\overline{\cdot}g$, and $\overline{S}(xg) = \overline{S}(x\,\overline{\cdot}\, g)=\overline{S}(g)\,\overline{\cdot}\, \overline{S}(x) = -x$. Notice that $\,\overline{\cdot}\,$ is braided-commutative and $\overline{S}$ is involutive, as we know from Theorem \ref{thm:RYDB}. Remark \ref{remark:involutive} implies that the braiding $\sigma_{H_4,H_4}\colon \alpha \ot \beta\mapsto (\alpha_1\rightharpoonup \beta_1)\ot (\alpha_2\leftharpoonup\beta_2)$ is involutive.

Observe that $\,\overline{\cdot}\,$ is also commutative, even if $\sigma^{\mathcal{YD}}_{H_4,H_4}$ does not coincide with the flip map $\tau$ (for instance, it is easy to see that $\sigma^{\mathcal{YD}}_{H_4,H_4}$ and $\tau$ differ on $x\ot x$). Notice moreover that $\,\overline{\cdot}\,$ is compatible with $\Delta$ in the braided sense, but not in the classical sense: indeed $(x\,\overline{\cdot}\, x)_1\ot (x\,\overline{\cdot}\, x)_2 = (k1-kg)\ot (k1-kg) \neq (k1-kg)\ot 1+2xg\ot g +1\ot (k1-kg) = (x_1\,\overline{\cdot}\, x_1)\ot (x_2\,\overline{\cdot}\, x_2)$. In particular, this Yetter--Drinfeld brace is not a Hopf brace.
\begin{remark}\label{classificationAgore}
Matched pairs of bialgebras $(H_4,H_4)$ on the Sweedler's Hopf algebra are classified by Bontea \cite[Theorem 2.1]{BC}. Other than the matched pairs arising from cotriangular structures, already considered above, there is only one more pair of actions on $H_4$ satisfying \eqref{matchedpair-1}--\eqref{matchedpair5}: this is the pair provided by the trivial actions $a\ot b\mapsto \epsilon(a)b$, $a\ot b\mapsto a\,\epsilon(b)$. 
Since $H_4$ is not commutative, the pair with the trivial actions does not satisfy \eqref{braided-commutativity}, and hence is not a matched pair of actions.
\end{remark}

We have, as a byproduct of Remark \ref{classificationAgore}, the following result:

\begin{proposition}
    The only matched pair of actions $(H_4,\rightharpoonup,\leftharpoonup)$ on the Sweedler's Hopf algebra $H_4$ is the matched pair defined in Table \ref{table:actions}.
\end{proposition}
\subsection{The Hopf algebras \emph{E}(\emph{n})} Let us suppose $\mathrm{char}(\Bbbk)\not=2$. A generalisation of the Sweedler's Hopf algebra is provided by the Hopf algebras $E(n)$ generated as algebras by $g, x_i$ for $i = 1,\ldots, n$, with relations
\[g^2 = 1,\;\; x_i^2 = 0,\;\; x_ig = -g x_i,\;\; x_ix_j = -x_j x_i,\;\;\text{ for all }i=1,\ldots, n,\]
endowed with the comultiplication $\Delta(1) = 1\ot 1$, $\Delta(g) = g\ot g$, $\Delta(x_i) = x_i\ot 1+g
\ot x_i$, the counit $\epsilon (g) = 1$, $\epsilon(x_i) =0$, and the antipode $S(g)=g$, $S(x_i) = x_i g$. 
The quasitriangular structures on $E(n)$ are classified by Panaite and Van Oystaeyen \cite{PanaiteOystaeyenEn}, and the coquasitriangular structures are obtained by duality (see e.g.\@ Carnovale and Cuadra \cite[p.\@ 256]{CarnovaleCuadraEn}):
\begin{align*}
    &\Rr_A\begin{pmatrix}
        1\ot 1& 1\ot g & 1\ot x_j & 1\ot x_j g \\
        g\ot 1& g\ot g & g\ot x_j & g\ot x_j g \\
        x_i\ot 1& x_i\ot g & x_i\ot x_j & x_i\ot x_j g \\
        x_ig\ot 1& x_ig\ot g & x_ig\ot x_j & x_ig\ot x_j g \end{pmatrix}= \begin{pmatrix}1 & 1 & 0 & 0 \\
    1 & -1 & 0 & 0 \\
    0 & 0 & A_{ij}& -A_{ij}\\
    0 & 0 & A_{ij}& A_{ij} \end{pmatrix},
\end{align*}
where $A=(A_{ij})_{ij}$ is an $n\times n$ matrix with entries in $\Bbbk$. It is easy to see that $\Rr_A$ is cotriangular if and only if $A$ is symmetric. Panaite and Van Oystaeyen \cite{PanaiteOystaeyenEn} conjectured that $R_A$ is triangular if and only if $A$ is symmetric; which was later proven by Carnovale and Cuadra \cite[Proposition 2.1]{CarnovaleCuadraEn}.

Assume $A_{ij} = A_{ji}$, thus $\Rr^{-1} = \Rr^\op$. We define a second multiplicative structure $\,\overline{\cdot}\,$ as in \eqref{cdot_R}. This is described in Table \ref{table:E(n)} (in particular, it is commutative). In terms of generators and relation, $(E(n), \,\overline{\cdot}\,,1)$ is the algebra generated by $g, x_i$ with relations $g \,\overline{\cdot}\, g = 1$, $x_i\,\overline{\cdot}\, x_j + x_j\,\overline{\cdot}\, x_i = 2A_{ij}(1-g)$, $x_i\,\overline{\cdot}\, g = g\,\overline{\cdot}\, x_i$. The action $\rightharpoonup$ is described in Table \ref{table:E(n)-action}.
\begin{table}[t]\centering
\begin{tabular}{c|cccc}
    $\,\overline{\cdot}\,$& $1$ & $g$ & $x_j$ & $x_j g$\\ \hline 
    $1$   & $1$& $g$& $x_j$& $x_j g$\\ 
    $g$   & $g$& $1$& $x_jg$&$x_j$  \\  
    $x_i$ & $x_i$& $x_ig$& $A_{ij}1- A_{ij}g+ x_ix_j$ & $A_{ij}g-A_{ij}1+x_ix_j g $\\  
    $x_ig$& $x_ig$& $x_i$& $A_{ij}g-A_{ij}1+x_ix_j g$ &  $A_{ij}1-A_{ij}g +x_ix_j $
    \end{tabular}
    \caption{Multiplication table for $\,\overline{\cdot}\,$ in $E(n)$.}\label{table:E(n)}
\end{table}
\begin{table}[t]
    \centering
    \begin{tabular}{c|cccc}
    $\rightharpoonup$& $1$ & $g$ & $x_j$ & $x_jg$\\ \hline
    $1$ & $1$ & $g$ & $x_j$ & $x_jg$ \\ 
    $g$ & $1$ & $g$ & $-x_j$ & $-x_jg$ \\ 
    $x_i$ & $0$ & $0$ & $A_{ij}1-A_{ij}g$ & $A_{ij}g-A_{ij}1$ \\ 
    $x_ig$ & $0$ & $0$ & $A_{ij}g-A_{ij}1$ & $A_{ij}1-A_{ij}g$ \\ 
    \end{tabular}\caption{Table for the action $\rightharpoonup$ in $E(n)$.}
    \label{table:E(n)-action}
\end{table}
The antipode $\overline{S}$, defined as in \eqref{S_R}, acts as follows: $\overline{S}(1) = 1$, $\overline{S}(g)=g$, $\overline{S}(x_i) = -x_i g = -x_i\,\overline{\cdot}\, g$.
\subsection{The Hopf algebra SL\textsubscript{\emph{q}}(2)}\label{subsection:Sq} Let $q\in \C^\times$, and consider, following Kassel \cite[{\S}$\,$IV.6]{kassel2012quantum}, the $\C$-algebra $\mathrm{SL}_q(2)$ generated by $a,b,c,d$ modulo the following relations:
\begin{align*}
    &ba = qa b, && ca = qac, && db = qbd,\\
    & dc = qcd,&& bc = cb,&& ad-da = (q^{-1}-q)bc,\\
    & && da- qbc = 1. && 
\end{align*}
This is known to be a Hopf algebra with 
\[\Delta\begin{pmatrix}
    a &b\\
    c & d
\end{pmatrix} = \begin{pmatrix}
    a &b\\
    c & d
\end{pmatrix}\ot \begin{pmatrix}
    a &b\\
    c & d
\end{pmatrix},\qquad \epsilon \begin{pmatrix}
    a &b\\
    c & d
\end{pmatrix} = \begin{pmatrix}
    1&0\\
    0&1
\end{pmatrix}, \qquad S\begin{pmatrix}
    a &b\\
    c & d
\end{pmatrix} = \begin{pmatrix}
    d & -qb\\
    -q^{-1}c & a
\end{pmatrix},
\]and, when $q\neq 1$, it is neither commutative nor cocommutative. The Hopf algebra $\mathrm{SL}_q(2)$ is coquasitriangular \cite[Corollary VIII.7.2]{kassel2012quantum}, with
\begin{align*}
\Rr&\begin{pmatrix}
    a\ot a & b \otimes b    &a \otimes b& b \ot a\\
    c\ot c & d\ot d& c\ot d & d\ot c\\
    a\ot c & b\ot  d& a\ot d& b\ot c\\
    c\ot a & d\ot b & c\ot b & d\ot a
\end{pmatrix}=q^{-1/2}\begin{pmatrix}
    q & 0 & 0 & 0\\
    0 & q & 0 & 0\\
    0 & 0 & 1 & q-q^{-1}\\
    0 & 0 & 0 & 1
\end{pmatrix},\\
 \Rr^{-1}&\begin{pmatrix}
    a\ot a & b \otimes b    &a \otimes b& b \ot a\\
    c\ot c & d\ot d& c\ot d & d\ot c\\
    a\ot c & b\ot  d& a\ot d& b\ot c\\
    c\ot a & d\ot b & c\ot b & d\ot a
\end{pmatrix}=q^{1/2}\begin{pmatrix}
    q^{-1} & 0 & 0 & 0\\
    0 & q^{-1} & 0 & 0\\
    0 & 0 & 1 & q^{-1}-q\\
    0 & 0 & 0 & 1
\end{pmatrix}.
\end{align*}
Notice that $\Rr^{-1}=\Rr^{\mathrm{op}}$ if and only if $q^{1/2}=q^{-1/2}$, if and only if $q=1$. In this case $\Rr=\Rr^{-1}=\Rr^{\mathrm{op}}$, and the previous matrices are the identities. The explicit form of the iterated $\Delta$ is here reported:
\begin{align*}
    (\Delta\ot \id)\Delta \, (a) &= a\ot a \ot a +a \ot b\ot c + b\ot  c\ot a+ b\ot d\ot c;\\
    (\Delta\ot \id)\Delta \, ( b) &= a\ot a\ot b + a\ot  b\ot d + b\ot c\ot b + b\ot d\ot d;\\
    (\Delta\ot \id)\Delta \, ( c) &= c\ot a\ot a +  c\ot b\ot c+ d\ot c\ot a+ d\ot d\ot c;\\
    (\Delta\ot \id)\Delta \, ( d) &=  c\ot a\ot b +  c\ot b\ot d + d\ot c\ot b +  d\ot d\ot d.
\end{align*}
We define an additional operation $\,\overline{\cdot}\,$ on $\mathrm{SL}_q(2)$ as in \eqref{cdot_R}. Using the iterated $\Delta$, we compute Table \ref{table:SLq(2)} for the multiplication $\,\overline{\cdot}\,$, and Table \ref{table:SLq(2)-action} for the left action $\rightharpoonup$. Hence we have that $\,\overline{\cdot}\,$ is commutative if and only if $q=q^{-1}$, i.e., if and only if $q=\pm1$. In particular, if $q=-1$, $\,\overline{\cdot}\,$ is commutative but $\Rr$ is not cotriangular.
\begin{table}[t]\centering
\begin{tabular}{c|cccc}
    $\,\overline{\cdot}\,$& $ a $ & $ b $ & $ c $ & $ d $\\ \hline
    $ a $ & $  a ^2$ & $q a  b  $ & $ q^{-1} a  c $ & $ a   d $ \\ 
    $ b $ & $ q^{-1} a  b  $ & $q^{-1} b ^2 $ & $(q^2-1)  a ^2 + (1-q^2)  a  d  + q  b  c $
 & $(q^3-q)  a  b  +  b  d  $ \\ 
    $ c $ & $q  a  c 
 $ & $ q b  c $ & $q^{-1} c ^2 $ & $  c  d  $ \\ 
    $ d $ & $ a  d  $ & $  b  d  $ & $(q^3-q) a  c  +  c  d 
 $ & $  d ^2$ \\ 
    \end{tabular}
    \caption{Multiplication table for $\,\overline{\cdot}\,$ in $\mathrm{SL}_q(2)$.}\label{table:SLq(2)}
\end{table}
\begin{table}[t]\centering
\begin{tabular}{c|cccc}
    $\rightharpoonup$& $ a $ & $ b $ & $ c $ & $ d $\\ \hline
    $ a $ & $  a  $ & $q^{-1}b  $ & $qc $ & $ d $ \\ 
    $ b $ & $ (1-q^{-2}) b  $ & $0 $ & $(q-q^{-1}) d + (q^{-1}-q)a $
 & $(1-q^2)b $ \\ 
    $ c $ & $ 0 $ & $ 0 $ & $0$ & $ 0$  \\ 
    $ d $ & $ a   $ & $  qb  $ & $q^{-1} c  $ & $ d $ \\ 
    \end{tabular}
    \caption{Table for the action $\rightharpoonup$ in $\mathrm{SL}_q(2)$.}\label{table:SLq(2)-action}
\end{table}
The algebra $(H,\,\overline{\cdot}\,,1)$ is generated by $a,b,c,d$ modulo the relations \begin{align*}&c\,\overline{\cdot}\, a = q^2 a\,\overline{\cdot}\, c, && b\,\overline{\cdot}\, d - d\,\overline{\cdot}\, b=(q^2-1)a\,\overline{\cdot}\, b,  \\
& a\,\overline{\cdot}\, b = q^2 b\,\overline{\cdot}\, a,&& b\,\overline{\cdot}\, c - c\,\overline{\cdot}\, b = (q^2-1)a \,\overline{\cdot}\, a - (q^2-1)a\,\overline{\cdot}\, d, \\
& a\,\overline{\cdot}\, d = d\,\overline{\cdot}\, a,&& d\,\overline{\cdot}\, c - c\,\overline{\cdot}\, d = (q^2-1) c\,\overline{\cdot}\, a,  
&\end{align*}
\[ a\,\bar{\cdot}\, d - q^{-2} c\,\bar{\cdot} \,b= 1. \]

We now define the transmuted antipode $\overline{S}(x):=x_{1}\rightharpoonup S(x_{2})$, as in \eqref{S_R}. 
One has
\begin{align*}
    &\overline{S}( a ) = q^{-2} d  + (1-q^{-2}) a , && \overline{S}( b ) = -q^{-2} b  ,\\
    & \overline{S}( c )= -q^{-2}  c , && \overline{S}( d ) =  a .
\end{align*}
Notice that $\overline{S}$ is involutive if and only if $q = \pm 1$. Hence, if $q=-1$, $\overline{S}$ is involutive but $\Rr$ is not cotriangular.
\subsection{An example in the class of Suzuki Hopf algebras} Let $\Bbbk$ be an algebraically closed field of characteristic $\mathrm{char}(\Bbbk)\neq 2$, as in Suzuki \cite{Suzuki}. The \emph{Suzuki Hopf algebra} $A^{\nu,\lambda}_{1,2}$, with parameters $\nu,\lambda\in\Bbbk$, is generated by $a,b,c,d$ subject to the relations \[
a^{2}=d^{2},\ b^{2}=c^{2}, \ a^{2}+\nu b^{2}=1,\ ab=ba=ac=ca=bd=db=cd=dc=0,\ cb=\lambda bc,\ ad=da.
\]
Let \[\mathbf{t}= \begin{pmatrix}a& b\\ c& d\end{pmatrix}, \]with $\mathbf{t}_{ij}$ denoting the entry on the $i$-th row and $j$-th column. The comultiplication, counit and antipode are respectively given by:
\[\Delta \mathbf{t} = \mathbf{t}\ot \mathbf{t},\;\; \epsilon (\mathbf{t}_{ij}) = \delta_{ij},\;\; S (\mathbf{t}_{ij}) = (\mathbf{t}_{ji})^{3}.\]
The map
\[ \Rr_{\alpha, \beta}\begin{pmatrix}
    a\ot a & a\ot b &a\ot c& a\ot d\\
    b\ot a&b\ot b&b\ot c &b\ot d\\
    c\ot a&c\ot b&c\ot c&c\ot d\\
    d\ot a& d\ot b&d\ot c&d\ot d
\end{pmatrix} = \begin{pmatrix}
    0&0&0&0\\
    0&\alpha &\beta&0\\
    0&\beta&\alpha&0\\
    0&0&0&0
\end{pmatrix}\]
defines a coquasitriangular structure whose convolution inverse is $\Rr^{-1}_{(\alpha,\beta)} = \Rr_{(\alpha^{-1},\beta^{-1})}$ (see \cite[proof of Proposition 2.1]{Suzuki}), which is cotriangular if $(\alpha,\beta) = (-1,-1)$ or $(\alpha,\beta)=(1,1)$, and $(\nu,\lambda) = (1,1)$ (see \cite[Proposition 3.10]{Suzuki}).

We define an algebra operation $\,\overline{\cdot}\,$ as in \eqref{cdot_R}, whose multiplication table is reported in Table \ref{table:Suzuki}. In particular, $\,\overline{\cdot}\,$ is commutative. The algebra $(H,\,\overline{\cdot}\,,1)$ is generated by $a,b,c,d$ modulo the relations \begin{align*} &a\,\overline{\cdot}\, a = d \,\overline{\cdot}\, d, \quad c \,\overline{\cdot}\, c = \alpha\beta^{-1}\lambda b \,\overline{\cdot}\, b,\quad c\,\overline{\cdot}\, b = b\,\overline{\cdot}\, c,\quad  a\,\overline{\cdot}\, d = d\,\overline{\cdot}\, a,\\
&a\,\overline{\cdot}\, b=b\,\overline{\cdot}\, a=a\,\overline{\cdot}\, c=c\,\overline{\cdot}\, a=b\,\overline{\cdot}\, d=d\,\overline{\cdot}\, b=c\,\overline{\cdot}\, d=d\,\overline{\cdot}\, c=0.\end{align*}The action $\rightharpoonup$ is reported in Table \ref{table:Suzuki-action}. The corresponding antipode $\overline{S}$, defined as in \eqref{S_R}, is identified by $\overline{S} (a) = \alpha^2\beta^2 d^3$, $\overline{S}(b) = \alpha^3\beta b^3$, $\overline{S}(c) = \alpha^3\beta c^3$, $\overline{S}(d) = \alpha^2\beta^2 a^3$.
\begin{table}[t]\centering
\begin{tabular}{c|cccc}
    $\,\overline{\cdot}\,$& $a$ & $b$ & $c$ & $d$\\ \hline
    $a$ & $\alpha^{-2}\beta^{-2}ad$ & $0$ & $0$ & $\alpha^{-2}\beta^{-2}a^2$ \\ 
    $b$ & $0$ & $\alpha^{-2}\beta^{-2}bc$ & $\alpha^{-3}\beta^{-1}b^2$ & $0$ \\ 
    $c$ & $0$& $\alpha^{-3}\beta^{-1}b^2$ & $ \alpha^{-1}\beta^{-3}\lambda bc$ & $ 0$ \\ 
    $d$ & $\alpha^{-2}\beta^{-2}a^2$ & $0$ & $0 $ & $ \alpha^{-2}\beta^{-2} ad$ \\ 
    \end{tabular}
    \caption{Multiplication table for $\,\overline{\cdot}\,$ in $A_{1,2}^{\nu,\lambda}$.}\label{table:Suzuki}
\end{table}
\begin{table}[t]\centering
\begin{tabular}{c|cccc}
    $\rightharpoonup$& $ a $ & $ b $ & $ c $ & $ d $\\ \hline
    $ a $ & $  d $ & $ \alpha^{-1}\beta c $ & $\alpha\beta^{-1} b $ & $ a $ \\ 
    $ b $ & $ 0 $ & $0 $ & $0 $
 & $0$ \\ 
    $ c $ & $ 0 $ & $ 0 $ & $0$ & $ 0$  \\ 
    $ d $ & $d  $ & $  \alpha\beta^{-1}c  $ & $\alpha^{-1}\beta b $ & $ a $ \\ 
    \end{tabular}
    \caption{Table for the action $\rightharpoonup$ in $A^{\nu,\lambda}_{1,2}$.}\label{table:Suzuki-action}
\end{table}
\vskip 15pt

We know by Theorem \ref{thm:RYDB} that the additional operation \eqref{def-cdot} is braided commutative if $\Rr$ is cotriangular.
In all the previous examples, when $\Rr$ is cotriangular \eqref{def-cdot} is also commutative. We could neither prove nor disprove that this is true in general, and thus the following problem is open:
\begin{problem}\label{problem:commutativity}
    Let $(H,\Rr)$ be a cotriangular Hopf algebra. Is it true that the additional operation \eqref{cdot_R} is commutative?
\end{problem}
Notice that the converse is false (see the example $SL_q(2)$ with $q = -1$). \vskip 15pt

\noindent\textbf{Acknowledgements.} The authors would like to thank A.\@ Ardizzoni for meaningful suggestions throughout all phases of this work, and L.\@ Vendramin and T.\@ Weber for valuable comments. This paper was written while the authors were members of the ``National Group for Algebraic and Geometric Structures and their Applications'' (GNSAGA-INdAM). They were partially supported by the Ministry for University and Research (MUR) within the National Research Project (PRIN 2022) ``Structures for Quivers, Algebras and Representations'' (SQUARE). The first-named author was partially funded by the Vrije Universiteit Brussel Bench Fee for a Joint Doctoral Project, grant number OZR4257, and partially supported through the FWO Senior Research Project G004124N. 
\bibliographystyle{acm}
\bibliography{refs}

\end{document}